\documentclass[11pt,fleqn,leqno]{siamltex}

\usepackage{amssymb,
amsmath,
graphicx,
hyperref,
}

\newtheorem{thm}{Theorem}[section]
\newtheorem{ex}[thm]{Example}
\newtheorem{que}[thm]{Question}
\newtheorem{de}[thm]{Definition}
\newtheorem{lm}[thm]{Lemma}
\newtheorem{prop}[thm]{Proposition}
\newtheorem{re}[thm]{Remark}

\newcommand{\ZZ}{{\mathbb Z}}
\newcommand{\RR}{{\mathbb R}}
\newcommand{\CC}{{\mathbb C}}
\newcommand{\GL}{\operatorname{GL}\nolimits}
\newcommand{\SL}{\operatorname{SL}\nolimits}
\newcommand{\Gr}{\operatorname{Gr}\nolimits}
\newcommand{\im}{\operatorname{im}}
\newcommand{\lieg}[1]{\mathrm{#1}}
\newcommand{\cha}{\operatorname{char}}

\newcommand{\AGL}{\mathrm{AGL}}
\newcommand{\cB}{\mathcal{B}}
\renewcommand{\phi}{\varphi}
\newcommand{\ph}{\phantom}
\newcommand*\samethanks[1][\value{footnote}]{\footnotemark[#1]}

\newcommand{\lsim}{\stackrel{\textstyle <}{\raisebox{-.6ex}{\(\sim\)}}}

\begin{document}

\title{Uniform determinantal representations}
\author{Ada Boralevi\thanks{%
Scuola Internazionale Superiore di Studi Avanzati, via Bonomea 265, 34136 Trieste, Italy,
\texttt{ada.boralevi@sissa.it}.
This author is member of GNSAGA and was partially supported by PRIN 2010-2011
project \lq\lq Geometria delle variet\`a algebriche\rq\rq and Universit\`a
di Trieste--FRA 2011 project \lq\lq Geometria e topologia delle variet\`a\rq\rq.}
\and
Jasper~van~Doornmalen\thanks{%
Department of Mathematics and Computer Science, TU Eindhoven, PO Box 513, 5600 MB, The Netherlands,
\texttt{m.j.v.doornmalen@student.tue.nl, j.draisma@tue.nl, m.e.hochstenbach@tue.nl}.
The first, third, and fourth author have been (partially) supported by NWO Vidi
research grants.}
\and
Jan~Draisma$^\dagger$\thanks{Faculteit Exacte Wetenschappen,
Afdeling Wiskunde,
Vrije Universiteit,
De Boelelaan 1081a,
1081 HV Amsterdam,
The Netherlands.}
\and
Michiel~E.~Hochstenbach\samethanks[2]
\and
Bor~Plestenjak\thanks{%
IMFM and Department of Mathematics, University of Ljubljana, Jadranska 19, 1000 Ljubljana, Slovenia,
\texttt{bor.plestenjak@fmf.uni-lj.si}.  This
author was partially supported by the ARRS Grant P1-0294.}
}


\maketitle

\begin{abstract}
The problem of expressing a specific polynomial as the determinant of
a square matrix of affine-linear forms arises from algebraic geometry,
optimisation, complexity theory, and scientific computing. Motivated
by recent developments in this last area, we introduce the notion of a
{\em uniform determinantal representation}, not of a single polynomial
but rather of all polynomials in a given number of variables and of
a given maximal degree. We derive a lower bound on the size of the
matrix, and present a construction achieving that lower bound up to a
constant factor as the number of variables is fixed and the degree grows.
This construction marks an improvement upon a recent construction due
to Plestenjak--Hochstenbach, and we investigate the performance of new
representations in their root-finding technique for bivariate
systems. Furthermore, we relate uniform determinantal representations
to vector spaces of singular matrices, and we conclude
with a number of future research directions.
\end{abstract}

\begin{keywords}
Determinantal representation, system of polynomial equations,
multiparameter matrix eigenvalue problem, space of singular matrices.
\end{keywords}

\begin{AMS}
13P15, 65H04, 65F15, 65F50.
\end{AMS}

\pagestyle{myheadings}
\thispagestyle{plain}
\markboth{BORALEVI ET AL}{UNIFORM DETERMINANTAL REPRESENTATIONS}

\section{Introduction and results} \label{sec:Intro}
Consider an $n$-variate polynomial of degree at most $d$:
\[ p=\sum_{|\alpha| \le d} c_{\alpha} x^{\alpha} \]
where $x:=(x_1,\ldots,x_n)$, $\alpha \in \ZZ_{\ge 0}^n$,
$|\alpha|:=\sum_i \alpha_i$, $x^{\alpha}:=\prod_i x_i^{\alpha_i}$,
and where each coefficient $c_\alpha$ is taken from a ground field $K$.
A {\em determinantal representation} of $p$ is an $N \times N$-matrix $M$
of the form
\[ M=A_0 + \sum_{i=1}^n x_i A_i, \]
where each $A_i \in K^{N \times N},$ with $\det(M)=p$. We call $N$
the {\em size} of the determinantal representation. Clearly, since the
entries of $M$ are affine-linear forms in $x_1,\ldots,x_n$, $N$ must be
at least the degree of $p$.

Determinantal representations of polynomials play a fundamental
role in several mathematical areas: from {\em algebraic geometry}
it is known that each plane curve ($n=2$) of degree $d$ over an
algebraically closed field $K$ admits a determinantal representation of
size $d$ \cite{Dickson,Dixon}. Over non-algebraically closed fields, and
especially when restricting to symmetric determinantal representations,
the situation is much more subtle \cite{IshitshukaIto}.  For larger $n$,
only certain hypersurfaces have a determinantal representation of size
equal to their degree \cite{Beauville,Dickson}. In {\em optimisation},
and notably in the theory of {\em hyperbolic polynomials} \cite{Wagner},
one is particularly interested in the case where $K=\RR$, $A_0$ is
symmetric positive definite, and the $A_i$ are symmetric. In this case,
the restriction of $p$ to any line through 0 has only real roots. For
$n=2$ the converse also holds \cite{HeltonVinnikov,LewisParriloRamana};
for counterexamples to this converse holding for higher $n$, see \cite{Branden}.
In {\em complexity theory} a central role is played by Valiant's
conjecture that the permanent of an $m \times m$-matrix does not admit a
determinantal representation of size polynomial in $m$ \cite{Valiant}. Via
the {\em geometric complexity theory} programme \cite{MulmuleySohoniI}
this leads to the study of polynomials in the boundary of the orbit of
the $N \times N$-determinant under the action of the group $\GL_{N^2}(K)$
permuting matrix entries. Recent developments in this field include the study of
this boundary for $N=3$ \cite{HuttenhainLairez} and the exciting negative
result in \cite{BurgisserIkenmeyerPanova} that Valiant's conjecture can
{\em not} be proved using occurrence obstructions proposed earlier in
\cite{MulmuleySohoniII}.

Our motivation comes from {\em scientific computing}, where determinantal
representations of polynomials have recently been proposed for efficiently
solving systems of equations \cite{BorMichiel}. For this application, it is
crucial to have determinantal representations not of a {\em single}
polynomial $p$, but rather of all $n$-variate polynomials of degree
at most $d$. Moreover, the representation should be easily computable
from the coefficients of $p$. Specifically, in \cite{BorMichiel} determinantal
representations are constructed for the bivariate case ($n=2$) in which
the entries of the matrices $A_0,\ldots,A_n$ themselves {\em depend
affine-linearly} on the coefficients $c_\alpha$. This is what we call a
{\em uniform determinantal representation} of the generic polynomial $p$
of degree $d$ in $n$ variables; see Section~\ref{sec:Problem} for a
precise definition.

\begin{ex}[The binary quadric] \label{ex:BinaryQuadric}
{\rm
The identity
\[ c_{00} + c_{10}x + c_{01}y + c_{20} x^2 + c_{11} xy +
c_{02} y^2= \det
\begin{bmatrix}
-x & 1 & 0 \\
-y & 0 & 1 \\
c_{00} & c_{10}+c_{20}x+c_{11}y & c_{01}+c_{02}y\\
\end{bmatrix}
\]
exhibits the matrix on the right as a uniform determinantal representation
of the generic bivariate quadric.
\hfill $\clubsuit$
}
\end{ex}

In applications, the matrix $M$ is used as input to algorithms
in numerical linear algebra that scale unfavourably with $N$,
such as a complexity of $O(N^6)$.
Consequently, we are led to consider the following fundamental question.

\begin{que}
What is the minimal size $N^*(n,d)$ of any uniform determinantal
representation of the generic polynomial of degree $d$ in $n$ variables?
\end{que}

A construction from \cite{BorMichiel} shows that
for fixed $n=2$ and $d \to \infty$ we have
$N^*(2,d) \le \frac{1}{4} \, d^2 + O(d)$;
this construction is reviewed in Section~\ref{sec:First}.
We improve the construction from \cite{BorMichiel} by giving a particularly elegant uniform
determinantal representations of bivariate polynomials of size $2d+1$ in
Example~\ref{ex:2d}, and of size $2d-1$ in Example~\ref{ex:2d1}.
In view of the obvious lower bound of $d$ this is
clearly sharp up to a constant factor for $d \to \infty$, although we do
not know where in the interval $[d,2d-1]$ the true answer lies.
We show in Section~\ref{sec:Numerics} how to use these small determinantal
representations of bivariate polynomials for solving systems of equations.
Before that, we focus on the asymptotic behaviour of $N^*(n,d)$
for fixed $n$ and $d \to \infty$. In this setting, we derive the
following result.

\begin{thm} \label{thm:Main}
For fixed $n \in \ZZ_{\ge 2}$ there exist positive constants $C_1,C_2$
(depending on $n$) such that for each $d \in \ZZ_{\ge 0}$ the smallest
size $N^*(n,d)$ of a uniform determinantal representation of the generic
polynomial of degree $d$ in $n$ variables satisfies
$C_1 d^{n/2} \le N^*(n,d) \le C_2 d^{n/2}$.
\end{thm}

We will also compare our results with previous constructions,
most notably with those by Quarez \cite[Thm.~4.4]{Quarez}, who proves the
existence of a symmetric
representation of size ${n+\lfloor \frac{d}{2} \rfloor} \choose n$.
For fixed $n$ and $d \to \infty$, \cite{Quarez} therefore has the asymptotic
rate $\sim d^n$, meaning that the results of this paper represent a clear improvement.
For fixed $d$ and $n \to \infty$, \cite{Quarez} leads to the asymptotic behavior
$\sim n^{\lfloor d/2 \rfloor}$, which is similar to our bounds;
we will discuss more details in Section~\ref{sec:Outlook}.

In Section~\ref{sec:Problem} we formalise the notion of uniform
determinantal representations, study their symmetries, and
derive some simple properties. In particular, we relate uniform
determinantal representations to spaces of singular $N \times N$-matrices. In
Section~\ref{sec:Singular} we briefly review some of the existing literature on
these singular spaces, and we prove that for $N>4$ there are
infinitely many equivalence classes of such objects; this poses an
obstruction to a ``brute-force'' approach towards finding lower bounds
on $N^*(n,d)$. In Section~\ref{sec:First} we present a first
construction, of which however the size is of the order of $d^n$,
rather than $d^{n/2}$, for $d \to \infty$.
In Section~\ref{sec:Second} we give a more efficient construction
and prove Theorem~\ref{thm:Main}.
In Section~\ref{sec:Small} we give upper bounds on $N^*(n,d)$
for small $n$ and $d$ and
determine $N^*(2,2)$ and $N^*(3,2)$ exactly. We extend representations
from scalar to matrix polynomials in Section~\ref{sec:MatPol}.
In Section~\ref{sec:Numerics} we give some numerical results that
show that for $n=2$ and small $d$ we get a competitive method
for computing zeros of polynomials systems.
Finally, in Section~\ref{sec:Outlook} we summarise our main conclusions
and collect some questions that arise naturally from our work.

\section{Problem formulation and symmetries} \label{sec:Problem}

In this section we give a formal definition of uniform determinantal
representations, and introduce a group that acts on such representations.
We also show that a uniform determinantal representation gives rise to
a vector space consisting entirely of singular matrices; such spaces
are the topic of next section.

Let $K$ be a field and fix $d,n \in \ZZ_{\ge 0}$. Let $F_d$ denote
the polynomials of degree at most $d$ in the polynomial ring $K[x_1,\ldots,x_n]$.
Furthermore, let $p_{n,d}$ be the generic polynomial of that degree, i.e.,
\begin{equation}\label{eq:polnd}
p_{n,d}=\sum_{|\alpha| \le d} c_\alpha x^\alpha,
\end{equation}
where $x:=(x_1,\ldots,x_n)$, $\alpha \in \ZZ_{\ge 0}^n$,
$|\alpha|:=\sum_i \alpha_i$, $x^{\alpha}:=\prod_i x_i^{\alpha_i}$,
and where we consider $c_\alpha$ as a variable for each $\alpha$.

\begin{de}
For $n,d \in \ZZ_{\ge 0}$, a {\em uniform determinantal representation}
of $p_{n,d}$ is an $N \times N$-matrix $M$ with entries from
$K[(x_1,\ldots,x_n),(c_\alpha)_{|\alpha| \le d}]$, of degree at most 1
in each of these two sets of variables, such that $\det(M)=p_{n,d}$. The
number $N$ is called the {\em size} of the determinantal representation.
\end{de}

To be explicit, we require each entry of $M$ to be a $K$-linear
combination of the monomials $1,x_i,c_\alpha,c_\alpha x_i,\
(i=1,\ldots,n, |\alpha| \le d)$. This means that we can decompose $M$
as $M_0+M_1$, where $M_0$ contains all terms in $M$ that do not
contain any $c_\alpha$, and where $M_1$ contains all terms in $M$ that
do. We will use the notation $M=M_0+M_1$ throughout the paper. When $n$
and $d$ are fixed in the context, we will also speak of a
uniform determinantal representation without reference to $p_{n,d}$. Our
ultimate aim is to determine the following quantity.

\begin{de}
For $n,d \in \ZZ_{\ge 0}$, $N^*(n,d) \in \ZZ_{>0}$ is the minimum among
all sizes of uniform determinantal representations of $p_{n,d}$.
\end{de}

This minimal size could potentially depend on the ground field $K$,
but the bounds that we will prove do not. Note that in the definition
of $N^*(n,d)$ we do not allow terms in $M$ of degree strictly larger
than one in the $c_\alpha$. Relaxing this condition to polynomial
dependence on the $c_\alpha$ might affect the exact value of $N^*$,
but it will not affect our bounds---see Remark~\ref{re:Poly}.

Given a uniform determinantal representation $M$ of size $N$, and given
matrices $g,h$ in $\SL_N(K)$, the group of determinant-one matrices
with entries in $K$,
the matrix $gMh^{-1}$ is another uniform determinantal representation
of $p_{n,d}$.  In this manner, the group $\lieg{SL}_N(K) \times
\lieg{SL}_N(K)$ acts on the set of uniform determinantal representations
of $p_{n,d}$. Moreover, there exist further symmetries, arising from affine
transformations of the $n$-space. Recall that these transformations form the
group $\AGL_n(K)=\GL_n(K) \ltimes K^n$ generated by invertible linear
transformations and translations.

\begin{lm} \label{lm:Affine}
The group $\AGL_n(K)$ acts on uniform determinantal representations of $p_{n,d}$.
\end{lm}

The statement of this lemma is empty without making the action
explicit, as we do in the proof.

\begin{proof}
Let $g \in \AGL_n(K)$ be an affine transformation of $K^n$, and expand
\[ p_{n,d}(g^{-1}x,c)=\sum_{|\alpha| \le d} c'_\alpha x^\alpha, \]
where the $c'_\alpha$ are linear combinations of the $c_\alpha$. More
precisely, the vector $c'$ can be written as $\rho(g)c$, where $\rho$
is the representation of $\AGL_n(K)$ on polynomials of degree at most $d$
regarded as a matrix representation relative to the monomial basis.

Now let $M=M(x,c)$ be a uniform determinantal representation of
$p_{n,d}$. Then
\[
\det(M(g^{-1}x,\rho(g)^{-1}c))=p_{n,d}(g^{-1}x,\rho(g)^{-1}c)=p_{n,d}(x,c),
\]
i.e., $M(g^{-1}x,\rho(g)^{-1}c)$ is another uniform determinantal
representation of $p_{n,d}$. The action of $g$ is given by $M \mapsto
M(g^{-1}x,\rho(g)^{-1}c)$.
\end{proof}

\begin{ex}[The binary quadric revisited]
{\rm
Take $n=d=2$ and the affine transformation
$g(x,y):=(y,x+1)$ with inverse $g^{-1}(x,y)=(y-1,x)$. We have
\[ p_{2,2}(g^{-1}(x,y))=
(c_{00} - c_{10} + c_{20})
+ (c_{01} - c_{11})x
+ (c_{10} - 2 c_{20})y
+ c_{02} x^2
+ c_{11} xy
+ c_{20} y^2.
\]
We find
{\footnotesize
\[
c'=
\begin{bmatrix}
c'_{00}\\c'_{10}\\c'_{01}\\c'_{20}\\c'_{11}\\c'_{02}
\end{bmatrix}
=
\begin{bmatrix}
1 & -1 & 0 & 1 & 0 & 0\\
0 & 0  & 1 & 0 &-1 & 0\\
0 & 1  & 0 &-2 & 0 & 0\\
0 & 0  & 0 & 0 & 0 & 1\\
0 & 0  & 0 & 0 & 1 & 0\\
0 & 0  & 0 & 1 & 0 & 0
\end{bmatrix}
\begin{bmatrix} c_{00}\\c_{10}\\c_{01}\\c_{20}\\c_{11}\\c_{02}
\end{bmatrix}
=\rho(g)c; \text{ }
\rho(g)^{-1}=
\begin{bmatrix}
1& 0& 1& 0& 0& 1\\
0& 0& 1& 0& 0& 2\\
0& 1& 0& 0& 1& 0\\
0& 0& 0& 0& 0& 1\\
0& 0& 0& 0& 1& 0\\
0& 0& 0& 1& 0& 0
\end{bmatrix}.
\]
}%
If we make the substitutions
\begin{align*}
c_{00} &\mapsto c_{00}+c_{01}+c_{02} &
c_{10} &\mapsto c_{01}+2c_{02} &
c_{01} &\mapsto c_{10}+c_{11} &
x &\mapsto y-1\\
c_{20} &\mapsto c_{02} &
c_{11} &\mapsto c_{11} &
c_{02} &\mapsto c_{20} &
y &\mapsto x
\end{align*}
in the uniform determinantal representation of
Example~\ref{ex:BinaryQuadric}, then we arrive at the matrix
\[
\begin{bmatrix}
1 - y& 1& 0\\
-x& 0& 1\\
c_{00} + c_{01} + c_{02}& c_{01} + c_{02} + c_{02}y + c_{11}x& c_{10}
+ c_{11} + c_{20}x
\end{bmatrix}
\]
whose determinant also equals $p_{2,2}$.
\hfill $\clubsuit$
}
\end{ex}

The action of the affine group will be used in Section~\ref{sec:Small}
to determine the exact value of $N^*(n,2)$ for $n=2$ and $3$. We now turn our
attention to the component $M_0$ of a uniform determinantal representation $M$.

\begin{lm} \label{lm:SingSpace}
For any uniform determinantal representation $M=M_0+M_1$ of
size $N$, the determinant of $M_0$ is the zero polynomial in
$K[x_1,\ldots,x_n]$. Moreover, at every point $\bar{x} \in K^n$, the rank
of the specialisation $M_0(\bar{x}) \in K^{N \times N}$ is exactly
$N-1$.
\end{lm}

\begin{proof}
The first statement follows from the fact that $\det(M_0)$ is the part
of the polynomial $\det(M)$ which is homogeneous of degree zero in the
$c_\alpha$; hence zero.

By specialising the vector $x$ of variables to a point $\bar{x} \in K^n$,
the rank of $M_0$ can only drop, so the rank of $M_0(\bar{x})$  is at most
$N-1$. However, if it were at most $N-2$, then after column operations on
$M$ by means of determinant-one matrices with entries in $K$ we may assume
that $M_0(\bar{x})$ has its last two columns equal to 0. This means that
all entries of $M(\bar{x})=M_0(\bar{x})+M_1(\bar{x})$ in these columns are linear
in the $c_\alpha$. This in turn implies that any term in the polynomial $\det
M(\bar{x})$ is at least quadratic in the $c_\alpha$. But on the other hand
$\det M(\bar{x})$ equals $p_{n,d}(\bar{x})$, which is a non-zero {\em linear}
polynomial in the $c_\alpha$ (nonzero since not every polynomial of
degree at most $d$ vanishes at $\bar{x}$). This contradiction implies that
the rank of $M_0(\bar{x})$ is $N-1$.
\end{proof}

\begin{lm} \label{lm:Dets}
If $M=M_0+M_1$ is a uniform determinantal representation of size
$N$, then $V \subseteq F_{N-1}$ spanned by the $(N-1) \times
(N-1)$-subdeterminants of $M_0$ satisfies $F_1 \cdot V \supseteq F_d$.
\end{lm}

Here, as in the rest of this paper, by the product of two spaces of
polynomials we mean the $K$-linear span of all the products.

\begin{proof}
Let $D_{ij}$ be the determinant of the submatrix of $M_0$ obtained
by deleting the $i$th row and the $j$th column. On the one hand,
$\det(M)=p_{n,d}$ is linear in the $c_\alpha$ by assumption, and on the
other hand, by expanding $\det(M)$ we see that the part that is homogeneous of
degree one in the $c_\alpha$ is
\[ \sum_{i,j}(-1)^{i+j} (M_1)_{ij} D_{ij};  \]
this therefore equals $p_{n,d}$. Hence any element $q$ of $F_d$
is obtained from the expression above by specialising the variables
$c_\alpha$ to the coefficients of $q$. Since each $(M_1)_{ij}$ is then
specialised to an element of $F_1$, we find $q \in F_1 \cdot V$.
\end{proof}

\begin{re} \label{re:Poly}
{\rm
Note that the proof above still applies if we allow determinantal
representations of the form $M_0+M_1+\cdots+M_e$ where $M_r$ is
homogeneous of degree $r$ in the $c_\alpha$ and affine-linear in the
$x_i$. Since our upper bound in
Section~\ref{sec:Second} builds directly on this lemma, the upper bound
holds in this more general setting, as well.
}
\end{re}

\section{Spaces of singular matrices} \label{sec:Singular}
Let $M=M_0+M_1$ be a uniform determinantal representation of $p_{n,d}$.
Writing $M_0=B_0+\sum_{i=1}^n x_i B_i$, Lemma~\ref{lm:SingSpace} implies
that the linear span $\langle B_0,\ldots,B_n \rangle_K \subseteq K^{N
\times N}$ consists entirely of singular matrices (and indeed that this
remains true when extending scalars from $K$ to an extension field). There
is an extensive literature on such {\em singular matrix spaces}; see,
e.g., \cite{DraismaBLMS,FillmoreLaurieRadjavi} and the references therein.
The easiest examples are the following.

\begin{de}
A subspace $\mathcal{A} \subseteq K^{N \times N}$ is called a {\em
compression space} if there exists a subspace $U \subseteq K^N$ with
$\dim(\langle u^T A \mid A \in \mathcal{A},u \in U\rangle_K)<\dim U$. We call
the space $U$ a {\em witness} for the singularity of $\mathcal{A}$.
\end{de}

Given any two subspaces $U,V \subseteq K^N$ with $\dim V=-1+\dim U$,
the space of all matrices which map $U$ into $V$ (acting on row vectors)
is a compression space with witness $U$. It is easy to see that these
spaces are inclusion-wise maximal among all singular spaces.

If $\mathcal{A}$ is a singular matrix space, then so is $g \mathcal{A}
h^{-1}$ for any pair $(g,h) \in \GL_N(K) \times \GL_N(K)$. We call the latter space {\em
conjugate} to the former.

\begin{ex} \label{ex:SmallSingSpace}
{\rm
For $N=2$, every singular matrix space is a compression space, hence
conjugate to a subspace of one of the two spaces
\[
\left\{\begin{bmatrix} * & * \\ 0 & 0 \end{bmatrix} \right\},
\left\{\begin{bmatrix} 0 & * \\ 0 & *\end{bmatrix}  \right\},
\]
where the $*$s indicate entries that can be filled arbitrarily.
A witness for the first space is the span $\langle e_2 \rangle$ of the
second standard basis vector, and a witness for the second space is $K^2$.

For $N=3$, there are four conjugacy classes of
inclusion-maximal singular matrix spaces, represented by the three
maximal compression spaces
\[
\left\{\begin{bmatrix} * & * & * \\ * & * & * \\ 0 & 0 &
0\end{bmatrix}  \right\},
\left\{\begin{bmatrix} * & * & * \\ 0 & 0 & * \\ 0 & 0 &
*\end{bmatrix}  \right\},
\left\{\begin{bmatrix} 0 & * & * \\ 0 & * & * \\ 0 & * &
*\end{bmatrix}  \right\},
\]
and the space of skew-symmetric $3 \times 3$-matrices \cite{EisenbudHarris}; the latter
is not a compression space.
For $N=4$, there are still finitely many (namely, $10$) conjugacy classes of
inclusion-maximal singular matrix spaces
\cite{EisenbudHarris,FillmoreLaurieRadjavi}, but this is not true
for $N \ge 5$, as Theorem~\ref{thm:Sing5} below shows.
This theorem is presumably folklore; we include a proof since
we have not been able to find a literature reference for it.
\hfill $\clubsuit$
}
\end{ex}

\begin{prop} \label{prop:Max}
Assume that $K$ is algebraically closed.  For any $m$ and $N\in \ZZ_{\ge
0}$ the locus $X_m$ in the Grassmannian $\Gr(m, K^{N \times N})$ of
$m$-dimensional subspaces of $K^{N \times N}$ consisting of all {\em
singular}
subspaces is closed in the Zariski topology. Moreover, the locus $U_m$ in
$X_m$ consisting of all {\em inclusion-wise maximal} singular
subspaces is open inside $X_m$.
\end{prop}

\begin{proof}
The first statement is standard. For the second statement, consider the
incidence variety
\[
Z:=\{(\mathcal{A},\mathcal{A'}) \in X_m \times X_{m+1} \mid
\mathcal{A} \subseteq \mathcal{A'} \} \subseteq X_m \times X_{m+1},
\]
which is a closed subvariety of $X_m \times X_{m+1}$.
The projection of $Z$ into $X_m$ is the complement of $U_m$, and it is
closed because $X_{m+1}$ is a projective variety.
\end{proof}

\begin{thm} \label{thm:Sing5}
Assume that $K$ is infinite and of characteristic unequal to two. For
$N \ge 5$ there are infinitely many conjugacy classes of inclusion-wise
maximal singular $N \times N$-matrix spaces.
\end{thm}

\begin{proof}
Take $N \ge 5$. For sufficiently general skew-symmetric matrices
$A_1,\ldots,A_N \in K^{N \times N}$ set $A:=(A_1,\ldots,A_N)$ and define
the space
\[ \cB_A:=\{(A_1x|\cdots|A_Nx) \mid x \in K^N\} \subseteq K^{N \times N}. \]
Each matrix in this space is singular, since for $x \neq 0$ we have
\[ x^T(A_1x|\cdots|A_Nx)=(x^TA_1x,\ldots,x^TA_Nx)=0. \]
In \cite{FillmoreLaurieRadjavi} it is proved that, for a specific
choice of the tuple $A$, the
space $\cB_A$ is maximal among the singular subspaces of $K^{N
\times N}$. By
Proposition~\ref{prop:Max}, $\cB_A$ is maximal for sufficiently general
$A$, as well (note that we may first extend $K$ to its algebraic closure
to apply the proposition). In the notation of that proposition,
we have a rational map
\[
\phi:S^N \dashrightarrow U_N,\quad A \mapsto \cB_A,
\]
where $S \subseteq K^{N \times N}$ denotes the subspace of
skew-symmetric matrices; the dashed arrow
indicates that the map is defined only in an open dense
subset of $S^N$. For any nonzero scalar $t$, $\phi(tA)=\phi(A)$.
We claim that, in fact, the general fibre of $\phi$ is indeed
one-dimensional. As the fibre dimension is semicontinuous, it suffices
to verify this at a particular point where $\phi$ is defined. We take
$A_i=E_{i,i+1}-E_{i+1,i}$ for $i=1,\ldots,N-1$ and $A_N$ general;
here $E_{ij}$ is the matrix with zeros everywhere except for a 1
at position $(i,j)$. Let $B \in S^N$; if $\phi(A)=\phi(B)$, then there
exists an invertible matrix $g \in \GL_N(K)$ such that
\[ (A_1 g x|\cdots|A_N g x)=(B_1 x |\cdots|B_N x) \]
for all $x$, so that $A_i g =B_i$. Using skew-symmetry of $A_i$ and
$B_i$, we find that $A_ig=g^T \! A_i$. Substituting our choice of $A_i$
for $i \in \{1,\ldots,N-1\}$ yields $g_{i,j}=g_{j,i}=0$ for all $j$
with $|i-j|>1$, $g_{i,i+1}=-g_{i,i+1}$, so $g_{i,i+1}=0$ since $\cha
K \neq 2$, and $g_{i,i}=g_{i+1,i+1}$. Hence $g$ is a scalar multiple
of the identity. It follows that the fibre of $\phi$ through $A$ is
one-dimensional as claimed.

Since $\dim S=\binom{N}{2}$, we have thus constructed an
$(N\binom{N}{2}-1)$-dimensional family inside $U_N$. Given any point
$\mathcal{A}$ in $U_N$, its orbit under $\GL_N(K) \times \GL_N(K)$ has dimension
at most $2 \, (N^2-1)$ (scalars act trivially). Now for $N=5$ we have
\[
N \, \binom{N}{2}-1=5\cdot 10 -1=49 \quad
\text{ and } \quad 2 \, (N^2-1)=48,
\]
so that we have found (at least) a one-parameter family of conjugacy
classes of singular spaces. For $N>5$ the difference between
$N\binom{N}{2}-1$ and $2 \, (N^2-1)$ is even larger.
\end{proof}

For large $N$ it seems impossible to classify maximal singular matrix
spaces. The construction above already gives an infinite number of
conjugacy classes, but there are many other sources of examples.
For instance, for infinitely many $N$ there exists a maximal singular
matrix space in $K^{N \times N}$ of constant dimension $8$, at least if
we assume that $K$ has characteristic 0 \cite{DraismaBLMS}.  On the other hand, if
the singular matrix space $\mathcal{A}$ has dimension at least $N^2-N$,
then it is a compression space with either a one-dimensional witness
or all of $K^N$ as witness \cite{Dieudonne} (and hence of dimension exactly
$N^2-N$). A sharpening of this result is proved in \cite{FillmoreLaurieRadjavi} (see also
\cite{deSeguinsPazzis}).

It should be noted that in many cases not even the
dimension of such singular matrix spaces is known, for fixed values of the size and rank of the matrices.
There is a considerable body of work devoted to giving
lower and upper bounds for such dimensions, both in the
case of bounded and constant rank, but these bounds are rarely sharp, see,
among many other references, \cite{Flanders, IlicLandsberg, Sylvester, Westwick}
and the more recent works on skew-symmetric matrices of constant rank
\cite{BoraleviFaenziMezzetti,ManivelMezzetti}.

Hence the fact that $M_0$ represents a singular matrix space of
dimension (at most) $n+1$ does not much narrow down our search for good
uniform determinantal representations, except in small cases discussed
in Section~\ref{sec:Small}. However, for our constructions in
Sections~\ref{sec:First} and~\ref{sec:Second} we will only use compression
spaces where the witness has dimension 1 or about $\frac{1}{2} N$, respectively;
and our lower bounds on $N^*(n,d)$ are independent of the literature on
singular matrix spaces.

\section{A first construction} \label{sec:First}

In this section we restrict our attention to determinantal
representations $M=M_0+M_1$ where $M_0$ represents a compression space
with a one-dimensional witness (or, dually by transposition, with a
full-dimensional witness). Under this assumption we will prove quite
tight bounds on the minimal size of a uniform matrix representation. The
following fundamental notion will be used throughout below.

\begin{de}
We say that a subspace $V \subseteq K[x_1,\ldots,x_n]$ is
{\em connected to $1$}
if it is nonzero and its intersections $V_e:=V \cap F_e$ satisfy
$F_1 \cdot V_e \supseteq V_{e+1}$ for each $e \ge 0$.
\end{de}

Note that this implies that $V_0=\langle 1 \rangle$. We borrow the
terminology from the theory of border bases \cite{LLMRT}, where a set
$S$ of monomials is called connected to $1$ if $1 \in S$ and each
nonconstant monomial in $S$ can be divided by some variable to obtain
another monomial in $S$. The linear span of $S$ is then
connected to $1$
in our sense. Translating monomials to their exponent vectors, we will
call a subset $S$ of $\ZZ_{\ge 0}^n$ {\em connected to
$0$} if it contains
0 and for each $\alpha \in S \setminus \{0\}$ there exists an $i$
such that $\alpha-e_i \in S$, where $e_i$ is the $i$-th
standard basis vector.

Let $V$ be a finite-dimensional subspace of $K[x_1,\ldots,x_n]$ connected
to 1.  Choose a $K$-basis $f_1,\ldots,f_m$ of $V$ whose total degrees
increase weakly.  For each $i=2,\ldots,m$ write
\[ f_i=\sum_{j<i} \ell_{ij} f_j \]
for suitable elements $\ell_{ij} \in F_1$. Let $M_V$ be the $(m-1)
\times m$-matrix whose $i$th row equals
\[ (-\ell_{i1},-\ell_{i2},\ldots,-\ell_{i,i-1},1,0,\ldots,0). \]
Note that $M_V$ depends on the choice of basis, but we suppress this
dependence in the notation, since the property of $M_V$ in the next
lemma does not depend on the choice of basis.

\begin{lm} \label{lm:Block}
The $K$-linear subspace of $K[x_1,\ldots,x_n]$ spanned by
the $(m-1) \times (m-1)$-subdeterminants of $M_V$ equals $V$.
\end{lm}

\begin{proof}
By construction, $M_V$ has rank $m-1$ over the field $K(x_1,\ldots,x_n)$
and satisfies $M_V \cdot (f_1,\ldots,f_m)^T=0$. By (a version
of) Cramer's rule, the kernel of $M_V$ is also spanned by
$(D_1,-D_2,\ldots,(-1)^{m-1}D_m)$ where $D_j$ is the determinant of the
submatrix of $M_V$ obtained by removing the $j$th column.  So these two
vectors differ by a factor in $K(x_1,\ldots,x_n)$.  Since $D_1=1=f_1$
we find that they are, in fact, equal. Hence $\langle D_1,\ldots,D_m
\rangle=V$ as claimed.
\end{proof}

We can now formulate our first general construction. This
generalises a construction from \cite{BorMichiel} to the multivariate case.

\begin{prop} \label{prop:Cons1}
Let $V$ be an $m$-dimensional space connected to $1$ and suppose that
$F_1 \cdot V \supseteq F_d$.  Then there exists a uniform determinantal
representation of size $m$ for the generic polynomial of degree $d$
in $n$ variables.
\end{prop}

\begin{proof}
Let $D_1,\ldots,D_m \in K[x_1,\ldots,x_n]$ be the
$(m-1)\times(m-1)$-subdeterminants of $M_V$. By $F_1 \cdot V=F_d$
and Lemma~\ref{lm:Block} we can find, for each monomial $x^\alpha$
of degree at most $d$, affine-linear forms
$l_{\alpha 1},\ldots,l_{\alpha m} \in F_1$ such that $x^{\alpha}=\sum_j
(-1)^{j-1} l_{\alpha j} D_j$. Define
\[ s:=\sum_{|\alpha| \le d} c_\alpha (l_{\alpha
1},\ldots,l_{\alpha m}), \]
a row vector of bi-affine linear forms in the $x_i$ and the
$c_{\alpha}$. Then, by Laplace expansion along the first row, we find
that the determinant of
\[ \begin{bmatrix} s\\M_V \end{bmatrix} \]
is the generic polynomial of degree $d$ in $x_1,\ldots,x_n$.
\end{proof}

\begin{ex}\label{ex:lin1}
{\rm
For $n=2$ the following picture gives a space $V$, connected to $1$
and spanned by the monomials marked with black vertices, such that $F_1
\cdot V=F_6$:
\begin{center}
\includegraphics{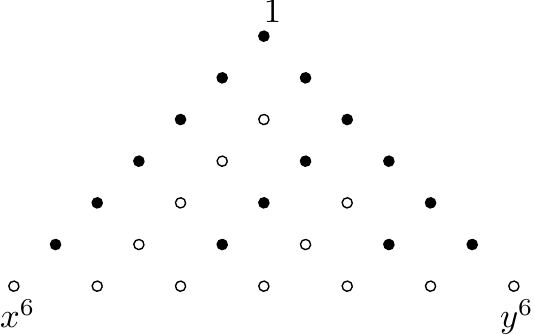}
\end{center}
This is the construction of \cite{BorMichiel}, which shows that there exists a
uniform determinantal representation of the generic bivariate polynomial of
degree $d$ of size $\frac{1}{4} \, d^2 + O(d)$ as $d \to \infty$.
\hfill $\clubsuit$
}
\end{ex}

The bivariate case generalises as follows.

\begin{thm} \label{thm:Cons1}
For fixed $n$, there exists a determinantal representation $M=M_0+M_1$ of
the generic $n$-variate polynomial of degree $d$ of size
$\frac{1}{n \cdot n!} \, d^n + O(d^{n-1})$
such that, moreover, the singular matrix space
represented by $M_0$ is a compression space with a one-dimensional
witness. Moreover, under this latter additional condition on $M_0$,
the bound is sharp.
\end{thm}

\begin{proof}
Note that $\dim F_d=\binom{n+d}{n}= \frac{1}{n!} \, d^n + O(d^{n-1})$. Hence by
Proposition~\ref{prop:Cons1} it suffices to show the existence of a
subspace $V \subseteq F_d$ connected to $1$ and such that $F_1 \cdot
V=F_d$, where $\dim V=\frac{1}{n} \, \dim F_d + O(d^{n-1})$. We will, in fact, show
that $V$ can be chosen to be spanned by monomials.

First, recall that there exists a lattice $\Lambda$ in $\ZZ^{n-1}$
such that $\ZZ^{n-1}$ is the disjoint union of $\Lambda$ and its cosets
$e_i+\Lambda$ for $i=1,\ldots,n-1$, namely, the root lattice of type $A_n$
generated by the rows of the $(n-1) \times (n-1)$-Cartan matrix
\[
\left[
\begin{array}{rrrrr}
2&-1& & &\\
-1&2&-1& &\\
 &\ddots&\ddots&\ddots&\\
 & &-1&2&-1\\
 & & &-1&2
\end{array}
\right],
\]
where the empty positions represent zeros
\cite[Planche 1]{Bourbaki}. In particular,
the index of $\Lambda$ in $\ZZ^{n-1}$ equals $n$. For example, if
$n=3$, here is the root lattice $\Lambda$ (in black) and its two
cosets (in gray and white):
\begin{center}
\includegraphics[scale=.5]{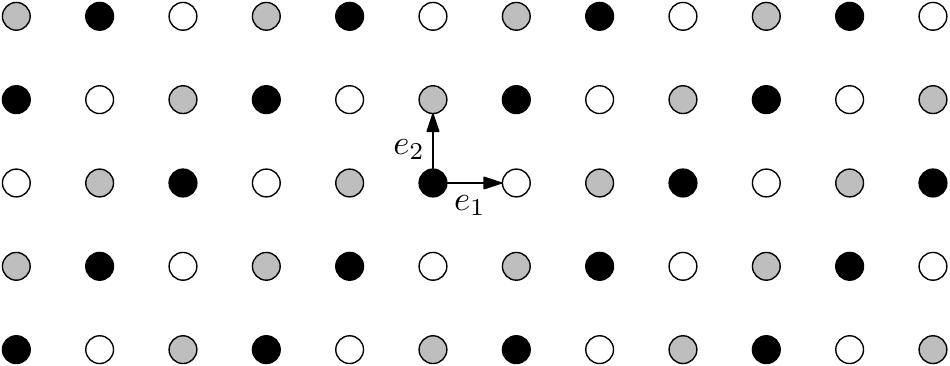}
\end{center}
Now let $\Delta_d$ be the simplex in $\RR^n$ with vertices
$0,de_1,\ldots,de_n$, for
$i=1,\ldots,n$ let $S_i$ be the set of lattice points in $\Delta_d$
that have $i$th coordinate zero, and set $S_0:=(\ZZ \times \Lambda)
\cap \Delta_d$. Define
\[ S:=S_1 \cup S_2 \cup \ldots \cup S_n \cup S_0,
\]
a subset of the lattice points in $\Delta_d$. We claim that $S$ is connected
to 0. Indeed, for each $i=1,\ldots,n$ the set $S_i$ is
connected to $0$,
and from each point $\alpha$ in $S_0$ one can walk within $S_0$
to $S_1$ by subtracting $\alpha_1$ times an $e_1$.

Next, we claim that for each $\alpha \in \Delta_d \cap \ZZ^n$ there
exists a $\beta \in S$ with $\alpha-\beta \in \{0,e_2,\ldots,e_n\}$.
Indeed, there is a (unique) $\beta'$ with this property in $\ZZ
\times \Lambda$. If this $\beta'$ has nonnegative entries, then set
$\beta:=\beta' \in S_0$. Otherwise, $\alpha$ itself has a
zero entry, say on the $i$th position, and we set $\beta:=\alpha \in S_i$.

Furthermore, for $i=1,\ldots,n$ the set $S_i$ contains $O(d^{n-1})$
vertices, and $S_0$ contains
$\frac{1}{n} \cdot \frac{1}{n!} \, d^n + O(d^{n-1})$
vertices. This concludes the construction--note that in the construction of
Proposition~\ref{prop:Cons1} the matrix $M_0$ has a zero
row, so that it represents a compression space with a
one-dimensional witness.

For sharpness, assume that $M=M_0+M_1$ is a uniform determinantal
representation of size $N$ such that the singular matrix space represented
by $M_0$ is a compression space with a one-dimensional witness. After a
choice of basis of $K^n$, we may assume that the first row of $M_0$
is identically zero; write $M_0=[0|M_0']^T$ accordingly. Let $u$ be the
first row of $M_1$ and write $M_1=[u|M_1']^T$. Then we have
\[
p=\sum_{|\alpha| \le d} c_\alpha x^\alpha = \det[u|M_0'+M_1'].
\]
Let $D_1,\ldots,D_N$ denote the $(N-1) \times (N-1)$ subdeterminants
of $M_0'$. By Lemma~\ref{lm:Dets}, the space $V$ spanned by these satisfies $F_1
\cdot V \supseteq F_d$. This already gives a lower bound of $V$ equal to
$d^n/((n+1)n!) + O(d^{n-1})$. To improve the $n+1$ in the denominator
into an $n$, we observe that by Cramer's rule, the map
\[
F_1^N \to K[x_1,\ldots,x_n],\ (\ell_1,\ldots,\ell_N) \mapsto \sum_i
(-1)^i \ell_i D_i
\]
has every column of $M_0'$ in its kernel. These
columns are linearly independent over $K$ (indeed over
$K(x_1,\ldots,x_n)$; see Lemma~\ref{lm:SingSpace}). We conclude that
\begin{equation} \label{eq:Ineq}
N \cdot \dim F_1 - (N-1) \ge \dim F_d,
\end{equation}
so that
\[ N \ge ((\dim F_d)-1)/n = d^n/(n \cdot n!) + O(d^{n-1}), \]
as desired.
\end{proof}

In the next section we derive a second
general construction of uniform determinantal representations, which we
use to prove Theorem~\ref{thm:Main}.

\section{A second construction} \label{sec:Second}

For a while, we believed that the uniform determinantal
representations of Theorem~\ref{thm:Cons1} were optimal.
But then we realised that
if one relaxes the condition that $M_0$ represent a compression space
with one-dimensional witness to the condition that $M_0$ represent
{\em some} compression space, smaller-size representations are possible.
The basic example is the following.

\begin{ex} \label{ex:2d}
{\rm
Let $p=\sum_{i+j \le 4} c_{ij} x^i y^j$ be the generic
polynomial of degree $d=4$ in $n=2$ variables. It has the
following uniform determinantal representation:
\begin{equation}\label{eq:repjan}
p=\det
\begin{bmatrix}
-x     & \ph-1  &      &      &      &     &     &     &     \\
      &   -x  & \ph-1  &      &      &     &     &     &     \\
      &      &   -x  & \ph-1  &      &     &     &     &     \\
      &      &      &   -x  & 1  &     &     &     &     \\
c_{00}&c_{10}&c_{20}&c_{30}&c_{40}&   -y &     &     &     \\
c_{01}&c_{11}&c_{21}&c_{31}&      & \ph-1 &   -y &     &     \\
c_{02}&c_{12}&c_{22}&      &      &     & \ph-1 &   -y &     \\
c_{03}&c_{13}&      &      &      &     &     & \ph-1 &   -y \\
c_{04}&      &      &      &      &     &     &     & \ph-1
\end{bmatrix},
\end{equation}
where the empty positions denote zeros. Let $M=M_0+M_1$ be the matrix
on the right-hand side. In this case, $M_0$ represents a
compression space with witness $U=\langle e_5,\ldots,e_9 \rangle_K$,
which is mapped into $\langle e_6,\ldots,e_9 \rangle_K$.

To verify the identity above without too many calculations, note that the
5 maximal subdeterminants of the $4 \times 5$-block with $x$'s are,
consecutively, $1,-x,x^2,-x^3,x^4$, and similarly for $y$. The matrix
obtained from $M$ by deleting the column corresponding to $x^i$ and the
row corresponding to $y^j$ has determinant $x^iy^j$.

This example extends to a uniform determinantal representation of size
$2d+1$ for the generic bivariate polynomial $p$ of degree $d$. We get
$p=\det(M)$, where
\[
M=(-1)^d\begin{bmatrix}M_x & 0\cr L & M_y^T\end{bmatrix},
\]
$M_x$ and $M_y$ are $d\times (d+1)$ matrices with 1 on the first
upper diagonal and $-x$ and $-y$ respectively on the main diagonal,
while $L$ is a $(d+1)\times (d+1)$ triangular matrix such that
$\ell_{ij}=p_{j-1,i-1}$ for $i+j\le d+2$ and 0 otherwise.
Note that we will slightly improve on the size $2d+1$ in Example~\ref{ex:2d1}.
\hfill $\clubsuit$
}
\end{ex}

Example~\ref{ex:2d} generalises as follows.

\begin{prop} \label{prop:Cons2}
Let $V,W \subseteq K[x_1,\ldots,x_n]$ be subspaces connected to $1$ such
that $F_1 \cdot V \cdot W \supseteq F_d$. Then there exists a uniform
determinantal representation of the generic $n$-variate polynomial of
degree $d$ of size $-1 + \dim V + \dim W.$
\end{prop}

\begin{proof}
Set $m_1:=\dim V$ and $m_2:=\dim W$. Consider the matrix
\[
M:=
\begin{bmatrix}
M_V & 0 \\ L & M_W^T
\end{bmatrix},
\]
with $M_V$ and $M_W$ the matrices of sizes $(m_1-1) \times m_1$
and $(m_2-1)\times m_2$ from Lemma~\ref{lm:Block}, and where
$L=(\ell_{ij})_{ij}$ is an $m_2 \times m_1$-matrix to be determined. Note
that the determinant of $M$ is linear in the entries of $L$.  Indeed,
setting $L=0$ yields the singular matrix $M_0$, so $\det(M)$ contains
no terms of degree 0 in the entries of $L$.
Furthermore, deleting from $M$ two or more
of the first $m_1$ columns from $M_V$, we end up with a matrix that
is singular since, when acting on rows, it maps the span of $\langle
e_1,\ldots,e_{m_1-1} \rangle$ into a space of dimension at most $m_1-2$,
so $\det(M)$ does not contain terms that are of degree $> 1$
in the entries of $L$.

Hence the determinant equals $\sum_{ij} \pm \ell_{ij} D_j E_i$ where
the $D_j$ are the maximal subdeterminants of $M_V$ and the $E_i$ are
the maximal subdeterminants of $M_W$. By Lemma~\ref{lm:Block} we have
$V=\langle D_1,\ldots,D_{m_1} \rangle_K$ and $W=\langle E_1,\ldots,E_{m_2}
\rangle_K$. Hence the assumption that $F_1 \cdot V \cdot W \supseteq F_d$
ensures that we can choose the $\ell_{ij} \in F_1$ in such a manner that
the determinant of $M$ equals the generic polynomial $p$.
\end{proof}

\begin{ex} \label{ex:2d1}
{\rm
Example~\ref{ex:2d} can be slightly improved to a representation of size
$2d-1$ by taking $V=\langle 1,x,\ldots,x^{d-1} \rangle$ and $W=\langle
1,y,\ldots,y^{d-1} \rangle$; note that, indeed, $F_1 \cdot V \cdot W
\supseteq F_d$.
A representation of size $2d-1$ for the polynomial $p$ from \eqref{eq:repjan} is
\begin{equation}\label{eq:repjan2}p=\det
\begin{bmatrix}
-x     & \ph-1  &      &      &      &     &         \\
      &   -x  & \ph-1  &      &      &     &         \\
      &      &   -x  & -1  &      &     &         \\
c_{00}&c_{10}&c_{20}&c_{30}+c_{40} x&  -y &      \\
c_{01}&c_{11}&c_{21}+c_{31}x & & \ph-1 &  -y &   \\
c_{02}+ c_{03} y&c_{12}+c_{22}x& &     &      & \ph-1 &   -y     \\
 c_{13} x+c_{04} y& &     &      &      &     & \ph-1 \\
\end{bmatrix}.
\end{equation}
 We do not know whether the factor 2 can be improved.
\hfill $\clubsuit$
}
\end{ex}

\begin{re}
{\rm
A representation of the form \eqref{eq:repjan2} can also be obtained
from the linearisations based on dual basis from \cite{Robol}. There,
linearisations of a univariate polynomial are presented that use the basis
of the form $\phi_i(x)\psi_j(x)$, where $\phi_i$ and $\psi_j$ are
polynomials. If we use the same approach for a bivariate polynomial
with the standard basis $\phi_i=x^i$ and $\psi_j=y^j$, we get a representation of the form
\eqref{eq:repjan2} up to permutations of rows and columns.
}
\end{re}

We will now prove our main theorem.

\begin{proof}[Proof of Theorem~\ref{thm:Main}, lower bound.]
For the lower bound on $N^*(n,d)$, let $M=M_0+M_1$ be a size-$N$ uniform
determinantal representation of the generic $n$-variate polynomial
of degree $d$. Let $D_{ij}$ be the $(N-1) \times (N-1)$-determinant
of the submatrix of $M_0$ obtained by deleting the entry at position
$(i,j)$. Then the image of the linear map
\[ \phi: F_1^{N \times N} \to K[x_1,\ldots,x_n],\
(\ell_{ij})_{i,j} \mapsto \sum_{i,j} (-1)^{i+j} \ell_{ij}
D_{ij} \]
contains $F_d$ (see Lemma~\ref{lm:Dets}).

We claim that $\phi$ has a kernel of dimension at
least $N(N-1)$. Indeed, fix any row index $i_0$. If the $D_{i_0,j}$ are
all zero, then we obtain an $(n+1)N$-dimensional subspace of $\ker \phi$
by setting all $\ell_{i,j}$ with $i \neq i_0$ equal to zero and choosing
the $\ell_{i_0,j} \in F_1$ arbitrarily. If they are not all zero, then
the $N-1$ rows with indices $i_1 \neq i_0$ are linearly independent over
$K(x_1,\ldots,x_n)$ and hence {\em a fortiori} over $K$. For
each such $i_1$ define $\ell^{(i_1)} \in F_1^{N \times N}$
by
\[ \ell^{(i_1)}_{ij}:=
\begin{cases}
(M_0)_{i_1,j} & \text{ if } i=i_0, \text{ and }\\
0 &\text{otherwise.}
\end{cases}
\]
Then, by Cramer's rule, we have $\phi(\ell^{(i_1)})=0$, and these $N-1$
vectors are linearly independent. Hence, for each $i_0$ we find a subspace
of $\ker \phi$ of dimension at least $N-1$, and these subspaces are
linearly independent. Thus we find that
\[ N^2(n+1)-N(N-1) = N^2 n + N  \ge \dim F_d =
\frac{d^n}{n!}+O(d^{n-1}), \]
from which the existence of $C_1$ follows.
\end{proof}

\begin{re}
{\rm
In the proof of the lower bound we have been a bit more careful than
strictly needed: without the discussion of the kernel it follows
that $N^2 \ge (\dim F_d)/(n+1)$. But one derives a better
constant (for $d\to \infty$) by using the kernel.
}
\end{re}

For the upper bound, we first give a simple construction for even $n$.
For odd $n$, a trickier analysis is needed (which also applies in
the even case); see below.

\begin{proof}[Proof of Theorem~\ref{thm:Main}, upper bound for even $n$]
Assume that $n=2m$ with $m \in \ZZ_{\ge 0}$. Let $V$ be the space of
polynomials in $x_1,\ldots,x_m$ of degree at most $d$, and let $W$ be
the space of polynomials in $x_{m+1},\ldots,x_n$ of degree at most $d$.
Then $V$ and $W$ are connected to $1$ and we have $F_1 \cdot V \cdot W \supseteq
F_d$, so that by Proposition~\ref{prop:Cons2} we have $N^*(n,d) \le -1 +
\dim V+\dim W$. Now compute
\[ \dim V=\dim W=\binom{m+d}{m}=\frac{d^{n/2}}{(n/2)!}+O(d^{m-1}). \]
This implies the existence of $C_2$ for even $n$.
\end{proof}

We now give a construction that works for all $n>2$, for which we thank
Aart Blokhuis.

\begin{proof}[Proof of Theorem~\ref{thm:Main}, upper bound for $n>2$]
For $i=0,1$ let $B_i \subseteq \ZZ_{\ge 0}$ denote the set of nonnegative
integers that can be expressed as $\sum_{j=0}^e b_j 2^{2j+i}$ with
$b_j \in \{0,1\}$, i.e., whose binary expansions have ones only at
even positions (for $i=0$, counting the least significant bit as
zeroeth position) or only at odd positions (for $i=1$). Observe that
$B_0+B_1=\ZZ_{\ge 0}$ and that both $B_0$ and $B_1$ contain roughly
$\sqrt{d}$ of the first $d$ nonnegative integers for every $d$--they
have ``dimension $1/2$''.  Now set $A_i:=B_i^n \subseteq \ZZ_{\ge 0}^n$
for $i=0,1$, so that $A_0+A_1=\ZZ_{\ge 0}^n$ and the number of elements
of $A_i$ intersected with a large box $[0,d]^n$ is roughly $d^{n/2}$.

Unfortunately, $A_0$ and $A_1$ are not connected to $0$.
However, we can connect them to 0 as follows. For a lattice point $\alpha \in A_i
\setminus \{0\}$ let $l$ be the minimum among the 2-adic valuations of
its entries, attained, say, by $\alpha_j$. Then
set $\tilde{\alpha}_j:=\alpha_j-2^l \in B_i$.  Setting the remaining
coordinates of $\tilde{\alpha}$ equal to those of $\alpha$ we have
$\tilde{\alpha} \in A_i$ and
\[ \|\alpha-\tilde{\alpha}\|_1=2^l. \]
We propose to add to $A_i$ the sequence
\[ \alpha-e_j,\alpha-2e_j,\ldots,\alpha-(2^l-1)e_j \]
to connect $\alpha$ to $\tilde{\alpha}$. We need to verify, however, that
in doing this, $A_i$ retains dimension $n/2$. The fraction of $\alpha$
in (a large box intersected with) $A_i$ for which the minimal valuation
is at least $l$ equals roughly $(2^{(-l+i)/2})^n$--after all, the
condition is that for each $j=1,\ldots,n$, $\alpha_j$ has zeros on
the first $(l-i)/2$ positions where it is allowed to have ones. Write
$l=i+2m$.  Thus by adding the sequences above, the total increase of $A_i$
is by a factor of at most
\[ \sum_{m=0}^\infty 2^{2m+i} (2^{-m})^n = \sum_{m=0}^\infty 2^{(2-n)m+i}. \]
This is a convergent series as $n>2$, and hence $A_i$ retains
dimension $n/2$.

We have thus constructed subsets $A_0,A_1 \subseteq \ZZ_{\ge 0}^n$ that
satisfy $A_0+A_1=\ZZ_{\ge 0}^n$, are connected to $1$, and that contain
roughly a constant times $d^{n/2}$ points in each box $[0,d]^n$. For
$i=0,1$ let $V_i$ be the space spanned by the monomials whose exponent
vectors lie in $A_i \cap [0,d]^n$. Then $V_1 \cdot V_2 \supseteq F_d$
and $V_1,V_2$ are connected to $1$, so by Proposition~\ref{prop:Cons2}
there exists a uniform determinantal representation for $p_{n,d}$ of
size $\dim V_1+\dim V_2=O(d^{n/2})$, as desired.
\end{proof}

\begin{re}
{\rm
The construction in the proof is by no means tight. For example, one could
also replace $\alpha_j$ by $\alpha_j-2^l+2^{l-2}+2^{l-4}+\ldots+2^i$,
which yields a shorter sequence to be added; and for large $l$ we have
also counted additional, shorter sequences since they also have valuation
larger than numbers smaller than $l$. We think that for even $n$ the
previous construction, subdividing the variables into two sets of equal
size, may lead to a better constant, but we have not verified this.
}
\end{re}

\begin{ex}
{\rm
Carrying out the construction in the proof for $A_0$ with $n=3$, always choosing
for $j$ the smallest index of a coordinate of $\alpha$ with minimal
valuation, we arrive at the following fractal-like structure (the circles indicate
the points of $A_0$, the black edges show that $A_0$ is
connected to $0$):
\begin{center}
\includegraphics[width=.4\textwidth]{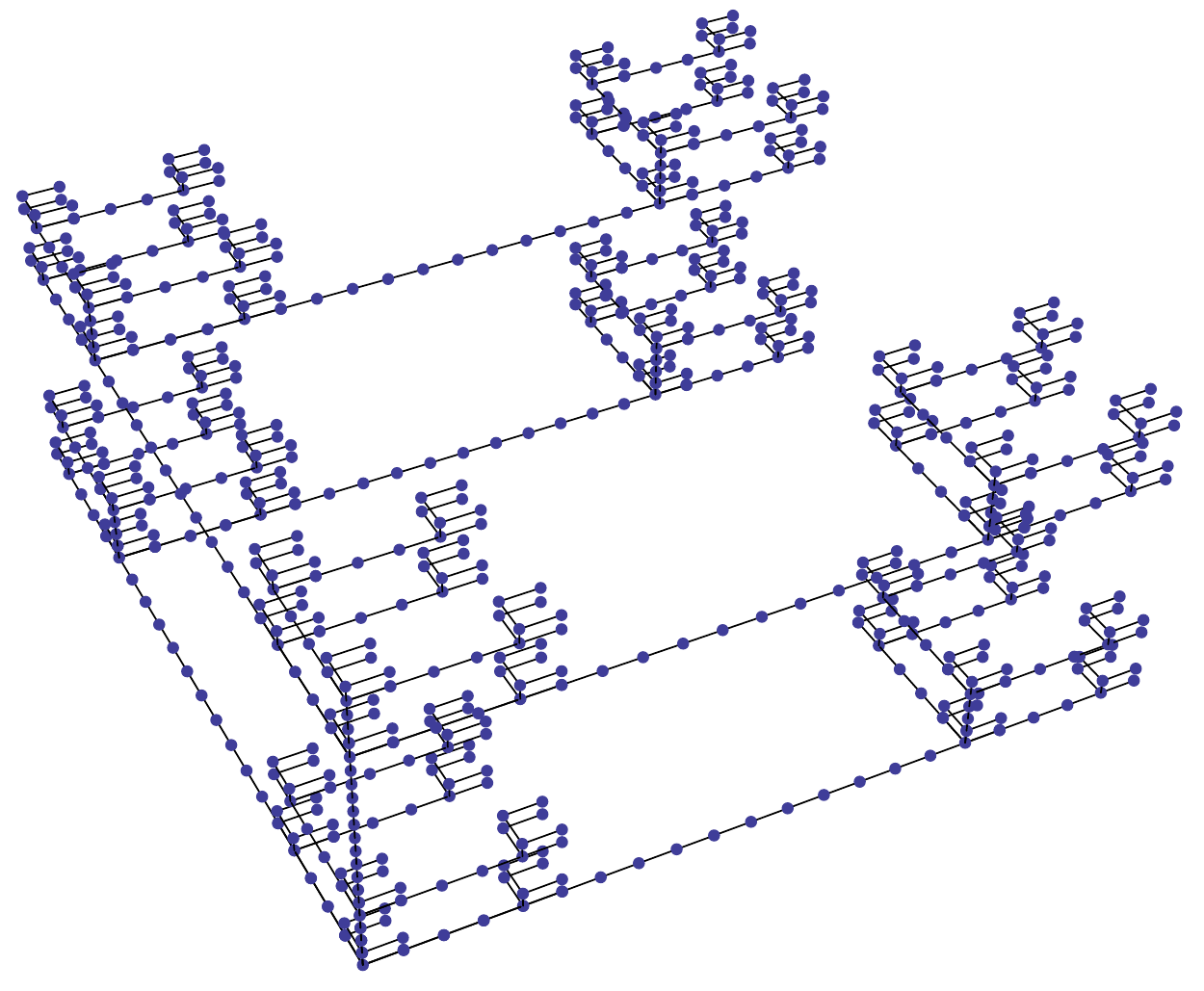}
\end{center}
}
\end{ex}

\section{Small $n$ and $d$} \label{sec:Small}

In this section we give several uniform representations of--to our knowledge--the smallest possible size
for cases where $n$ and $d$ are small. We start with the two cases
where we can compute $N^*(n,d)$ exactly.

\begin{prop}
$N^*(2,2)=3$.
\end{prop}

\begin{proof}
Taking $V=\langle 1,x,y \rangle$ in Proposition~\ref{prop:Cons1}
we see that $N^*(2,2) \le 3$; this is the representation
of Example~\ref{ex:BinaryQuadric}. Suppose that a uniform
determinantal representation $M=M_0+M_1$ of size $N=2$ exists. Then,
by Example~\ref{ex:SmallSingSpace}, after acting with $\SL_2(K) \times
\SL_2(K)$ and transposing if necessary, we may assume that the singular
space represented by $M_0$ is a compression space with a one-dimensional
witness. But then \eqref{eq:Ineq} reads
\[ 2 \cdot 3 - 1 = N \cdot \dim F_1-(N-1) \ge \dim F_2 = 6, \]
a contradiction. Hence $N^*(2,2)=3$.
\end{proof}

\begin{prop}
$N^*(3,2)=4$.
\end{prop}

\begin{proof}
Taking $V=\langle 1,x,y,z \rangle$ in Proposition~\ref{prop:Cons1}
we see that $N^*(3,2)\le 4$. Suppose that a uniform representation of
size $N=3$ exists. Up to transposition, there are three possibilities
for the singular space $\mathcal{A}$ represented by $M_0$; see
Example~\ref{ex:SmallSingSpace} (where the third
compression space is conjugate to the transpose of the
first):
\begin{enumerate}
\item
Assume that $\mathcal{A}$ is a compression space with a
one-dimensional witness, so that after acting with $\SL_3(K) \times
\SL_3(K)$ we have
\[ M_0=\begin{bmatrix}
0 & 0 & 0 \\
* & * & *\\
* & * & *
\end{bmatrix}.
\]
Let $D_j$ denote the determinant of the minor of $M_0$ obtained by
deleting the first row and $j$th column. Then the linear map
\[ \Omega:F_1^3 \mapsto K[x,y,z],\ (l_1,l_2,l_3) \mapsto
l_1D_1-l_2D_2+l_3D_3 \]
has $F_2 \subseteq \im \Omega$. Now inequality~\eqref{eq:Ineq} reads
\[ 3 \cdot 4 - 2 = N \cdot \dim F_1-(N-1) \ge \dim F_2 = 10, \]
which holds with equality. This means that, in fact, $\im \Omega$
must {\em equal} $F_2$. In particular, $D_1,D_2,D_3$ must all be of
degree one (or else $\im \Omega$ would contain cubic polynomials). The
image of $\Omega$ depends only on the span $V:=\langle D_1,D_2,D_3
\rangle \subseteq F_1$. If $1 \not \in V$, then there exists an
affine transformation in $\AGL_3(K)$ that maps $V$ into a subspace of
$\langle x,y,z \rangle$.  Then $1 \not \in F_1 \cdot V=\im \Omega$,
a contradiction. If $1 \in V$, then after an affine transformation we
find $\langle 1 \rangle \subseteq V \subseteq \langle 1,x,y \rangle$. In
that case, $z^2 \not \in F_1 \cdot V$, another contradiction.

\item Assume that $\mathcal{A}$ is a compression space with a
two-dimensional witness, so that after row and column operations we
have
\[
M_0=\begin{bmatrix}
0&0&q \\
0&0&r \\
s&t&*
\end{bmatrix},
\]
where $q,r,s,t \in F_1$. Write $M_1=(m_{ij})_{ij}$. Using that $\det(M)$
is assumed to be linear in the $c_\alpha$s, we find that
\[ \det(M)=-m_{11} rt + m_{12} rs + m_{21} qt - m_{22} qs. \]
Consequently, setting $V_1:=\langle q,r \rangle$ and $V_2:=\langle s,t
\rangle$, we have $F_1 \cdot V_1 \cdot V_2 \supseteq F_2$. If $1 \not \in
V_1$, then by acting with a suitable element of $\AGL_3(K)$ we achieve
that $V_1 \subseteq \langle x,y,z \rangle$. But then $F_1 \cdot V_1
\cdot V_2 \not \ni 1$. The same applies when $1 \not \in V_2$. On the
other hand, if $1 \in V_1 \cap V_2$, then by an element in $\AGL_3(K)$
we achieve that $\langle 1 \rangle \subseteq V_1,V_2 \subseteq \langle
1,x,y\rangle$. In that case, $z^2 \not \in F_1 \cdot V_1 \cdot V_2$.

\item Finally, assume that $\mathcal{A}$ is conjugate to a space of
skew-symmetric matrices, so that after conjugation
\[ M_0=
\left[
\begin{array}{rrr}
0 & q & r \\
-q& 0 & s \\
-r&-s & 0
\end{array}
\right]
\]
where $q,r,s \in F_1$. Set $V:=\langle q,r,s \rangle \subseteq F_1$. Then
the space spanned by the $2 \times 2$-determinants of $M_0$ is $V \cdot
V$, of dimension at most $6$. Moreover, we have $F_1 \cdot V \cdot V
\supseteq F_2$. If $1 \not \in V$, then by acting with $\AGL_3(K)$ we
achieve that $V \subseteq \langle x,y,z \rangle$, and hence $1 \not \in
F_1 \cdot V \cdot V$. If, on the other hand, $1 \in V$, then we achieve
that $\langle 1 \rangle \subseteq V \subseteq \langle 1,x,y \rangle$,
and $z^2 \not \in F_1 \cdot V \cdot V$.
\end{enumerate}
In each of these cases we arrive at a contradiction. Consequently,
$N^*(3,2)=4$ as claimed.
\end{proof}

The proofs above use the classification of spaces of small singular
matrices in an essential manner, as well as the action of $\AGL_n(K)$
on uniform determinantal representations. We conjecture that
$N^*(4,2)=5$, and that this can still be proved in the same manner,
using the classification of $4 \times 4$-singular matrix spaces from
\cite{FillmoreLaurieRadjavi}. But as Theorem~\ref{thm:Sing5} shows,
fundamentally new ideas will be needed to prove lower bounds in larger
situations.

For some pairs of small $n$ and $d$ we now give the smallest
uniform representations that we have been able to find.
For the constructions we use Proposition~\ref{prop:Cons2} with
subspaces $V,W\subseteq K[x_1,\ldots,x_n]$ spanned by the monomials and
connected to $1$. First, we give
in Table~\ref{tab:minnd} the minimal sizes known to us of
uniform determinantal representations for some small values
of $n$ and $d$.

\begin{table}[!htbp]
\caption{Minimal known sizes of uniform determinantal representations
we have been able to construct for $n$-variate polynomials of degree $d$;
cf.~Table~\ref{tab:wv}.\label{tab:minnd}}
\begin{center}
{\footnotesize
\begin{tabular}{c|ccccccccc} \hline \rule{0pt}{2.3ex}%
$n$ & $d=2$ & $d=3$ & $d=4$ & $d=5$ & $d=6$ & $d=7$ & $d=8$ & $d=9$ \\ \hline \rule{0pt}{2.3ex}%
2 & 3 & \ph{1}5 & \ph{1}7 & \ph{1}9 & 11& 13& 15& 17 \\
3 & 4 & \ph{1}7 & 10  & 14& 18& 22& 27& 34 \\
4 & 5 & \ph{1}9 &14 & 19& 26& 34& 44 \\
5 & 6 &11 &18 & 26& \\
6 & 7 &13 &22 & 33& \\
7 & 8 &15 &27 & 39& \\
8 & 9 &17 &32 &  \\ \hline
\end{tabular}
}
\end{center}
\end{table}

The corresponding representations for the entries
in Table~\ref{tab:minnd} for $n=2$, which are of size $2d-1$, are given in
Example~\ref{ex:2d1}.
For $d=2$ we take $V= \langle 1,x_1,\ldots,x_n \rangle$ and
$W= \langle 1 \rangle$,
therefore $N^*(n,2)\le n+1$, while for $d=3$ we can
take $V=W= \langle 1,x_1,\ldots,x_n \rangle$, and hence $N^*(n,3)\le
2n-1$.
In Table~\ref{tab:wv}
we give sets $V$ and $W$ for the remaining nonzero entries in Table~\ref{tab:minnd}.
The subspaces $V$ and $W$ have the form $V=V_0\cup V_1$ and
$W=W_0\cup W_1$, where
\begin{equation}\label{eq:vw0}
\begin{matrix}
  V_0=\langle1,x_1,\ldots,x_n,\ldots,x_1^e,\ldots,x_n^e\rangle,\\[0.4em]
  W_0=\langle1,x_1,\ldots,x_n,\ldots,x_1^f,\ldots,x_n^f\rangle
\end{matrix}
\end{equation}
for $e=\lceil (d-1)/2\rceil$ and $f=\lfloor (d-1)/2\rfloor$, which yields $d-1=e+f$.
For clarity and brevity, the variables $x,y,z,w,u,v,q,s$ in Table~\ref{tab:wv}
stand for $x_1,\ldots,x_8$, respectively.

\begin{table}[!htbp]
\begin{center}
\caption{List of monomials in $V_1$ and $W_1$ that,
together with $V_0$ and $W_0$ from \eqref{eq:vw0},
lead to uniform representations of sizes as in Table~\ref{tab:minnd}.\label{tab:wv}}
{\footnotesize
\begin{tabular}{ll|ll} \hline \rule{0pt}{2.3ex}%
$n$ & $d$ & $V_1$ & $W_1$\\
\hline \rule{0pt}{2.3ex}%
3 & 4 & $-$ & $-$\\
3 & 5 & $-$ & $xy$\\
3 & 6 & $-$ & $xy,\, x^2y$\\
3 & 7 & $-$ & $x^2y,\, y^2z,\, z^2 x$\\
3 & 8 & $-$ & $x^2y,\, y^2z,\, z^2x,\, x^2y^2,\, z^2w^2$ \\
3 & 9 & $x^3y,\,y^3z,\,z^3x$ & $x^2y,\,x^2z,\,y^2z,\,x^2y^2,\,x^2z^2,\,y^2z^2$\\[2mm]
4 & 4 & $-$ & $xy$\\
4 & 5 & $-$ & $xy,\, zw$\\
4 & 6 & $x^2y,\, y^2z,\, z^2w$ &
  $xy,\, x^2y,\, z w$\\
4 & 7 & $x^2y,\, y^2x,\, z^2w,\, w^2x,\, xy$ &
$x^2z,\, xz^2,\,y^2w,\, yw^2$\\
4 & 8 & $x^2y,\, x^2y^2,\, z^2x,\, x^3 y,\, y^3z,\,
z^3w,\,w^3x$ &
$ xy,\,xyz,\, xyw,\, y^2z,\, z^2w,\,
w^2x,\, w^2y,\, x^2z$\\[2mm]
5 & 4 & $-$ & $ xy,\, zw$\\
5 & 5 & $xy,\, yz,\, zw$ & $wu,\, xu$\\[2mm]
6 & 4 & $-$ & $xy,\, zw,\, uv$\\
6 & 5 & $xy,\, zw,\,uv,\,wy$ & $yz,\, wu,\, xv,\, xz$\\[2mm]
7 & 4 & $-$ & $xy,\, zw,\, uv,\, xq,\, yq$\\
7 & 5 & $xy,\, z w,\, u v,\, w y,\, q u$ &
$y z,\, w u,\, v q,\, x z,\, w x$\\[2mm]
8 & 4 & $-$ & $x y,\, y z,\, x z,\, w u,\, w v,\, u v,\, q s$ \\ \hline
\end{tabular}}
\end{center}
\end{table}

\begin{ex} \label{ex:turan}
{\rm
To show how things get complicated, let us
consider the construction for $d=4$. We take $V=
\langle 1,x_1,\ldots,x_n,x_1^2,\ldots,x_n^2 \rangle$ and
\[W= \langle
1,x_1,\ldots,x_n,x_{\alpha_1}x_{\beta_1},\ldots,x_{\alpha_m}x_{\beta_m}
\rangle,\]
where $1\le \alpha_i<\beta_i\le n$ and $m$ is as small as possible. If we take all possible pairs
$x_{\alpha}x_\beta$, then clearly $F_1 \cdot V \cdot W \supseteq F_4$,
while on the other hand, when $m=0$, $F_1 \cdot V \cdot W$ does not contain
any monomials of the form
\begin{equation}
\label{eq:ijkl}
x_i \, x_j \, x_k \, x_\ell
\end{equation}
for $1\le i<j<k<\ell\le n$.
We need a minimal set of $x_\alpha x_\beta$ to cover all possible
monomials \eqref{eq:ijkl}, which is related to the following covering problem.

Given positive integers $r\le k\le n$, we say that a system $S$ of $r$-subsets
(called blocks) of $\{1,\dots, n\}$ is called
a Tur\'an $(n,k,r)$-system if every $k$-subset of $\{1,\dots, n\}$ contains at
least one block from $S$ \cite{Sidorenko}. The minimum size of $S$
is called the Tur\'an number $T(n,k,r)$.

In our case, additional terms $x_{\alpha_1}x_{\beta_1},\ldots,x_{\alpha_m}x_{\beta_m}$ form
a Tur\'an $(n,4,2)$-system. While for most cases
only upper and lower bounds for $T(n,k,r)$ are known, Tur\'an proved
that
\begin{equation}\label{eq:Turan}
T(n,4,2)=mn-3 \, {m(m+1)\over 2},
\end{equation}
where $m=\lfloor n/3\rfloor$. To obtain the minimal set one has to divide $\{1,\ldots,n\}$
into three nearly equal groups (their sizes do not differ for more than one)
and then take all pairs $x_\alpha x_\beta$
such that $\alpha$ and $\beta$ belong to the same group.
As a result, such construction gives a uniform representation
of size $N$, where $N=\frac{1}{6} \, n^2+O(n)$,
which therefore implies $N^*(n,4) \le \frac{1}{6} \, n^2+O(n)$.
\hfill $\clubsuit$
}
\end{ex}

\section{Matrix polynomials} \label{sec:MatPol}
Suppose that we have a uniform representation $M$ of
$p_{n,d}$ as in \eqref{eq:polnd}, and write
\begin{equation}\label{eq:linM}
M=M_0+M_1=M_0 + \sum_{|\alpha| \le d} c_\alpha M_\alpha,
\end{equation}
where each $M_\alpha \in F_1^{N \times N}$.
Now consider the matrix polynomial (cf.~\eqref{eq:polnd})
\[ P_{n,d}=\sum_{|\alpha| \le d} x^\alpha C_\alpha , \]
where $C_\alpha$ is a $k\times k$ matrix. We will show that under
certain assumptions we can construct from $M$ a matrix
$\widetilde{M}$ that represents $P_{n,d}$ in the sense
that
$\det(\widetilde M)=\det(P_{n,d})$.
We obtain $\widetilde M$ from $M$ in the following way. Each element of the form
$\alpha +\beta x +\gamma y$ is replaced by the $k\times k$ matrix
$(\alpha+\beta x + \gamma y)I_k$, where $I_k$ is the $k \times k$ identity,
and each $c_\alpha$ is replaced by the matrix $C_\alpha$.
%

\begin{thm}\label{thm:matpollin} Let \eqref{eq:linM} be a uniform representation
of the generic polynomial \eqref{eq:polnd} of degree $d$ in $n$ variables
and assume that there exist matrices $Q$ and $Z$, whose elements are polynomials in
$x_1,\ldots,x_n$, such that $\det(Q) = \det(Z) = 1$, and
$QMZ$ is a triangular matrix with one diagonal element equal to $p_{n,d}$
and all other diagonal elements equal to 1.
Then
\[\widetilde M = M_0\otimes I_k+
\sum_{|\alpha|\le d}M_\alpha\otimes C_\alpha\]
is a representation for the matrix polynomial $P_{n,d}$, i.e.,
$\det(\widetilde M)=\det(P_{n,d})$.
\end{thm}
\begin{proof}
It is easy to see that $(Q\otimes I_k)\widetilde M(Z\otimes I_k)$
is a block triangular matrix with one diagonal block $P_{n,d}$ while
all other diagonal blocks are equal to $I_k$. Since $\det(Q\otimes I_k)=\det(Z\otimes I_k)=1$,
it follows that $\det(\widetilde M)=\det(P_{n,d})$.
\end{proof}

\begin{ex}
{\rm
Theorem~\ref{thm:matpollin} applies to the uniform representation~\eqref{eq:repjan}.
Indeed, take
\[
Q=\begin{bmatrix}
1 &  &  & & & & &\cr
& \ddots  & & & & & & \cr
& &  1 & & & & & \cr
& &  & 1 & y& y^2 &y^3 &y^4 \cr
& &  &  & 1 & y & y^2& y^3\cr
& &  & &  & 1 & y& y^2\cr
& &  & & & & 1 & y\cr
& &  &  & & & & 1\end{bmatrix},
\qquad
Z=\begin{bmatrix} 1 &  &  &  &  & & & \cr
x & 1 & & &  & & & \cr
x^2 & x & 1 &  &  &  & & \cr
x^3 & x^2 & x& 1 &  & &  & \cr
x^4 & x^3 & x^2 & x & 1 & & &  \cr
& & & & & 1 & &  \cr
& & & & & & \ddots & \cr
& & & & & & & 1\end{bmatrix},
\]
then
\[
QMZ=
\begin{bmatrix}
     &   1  &      &      &      &     &     &     &     \\
      &     &   1  &      &      &     &     &     &     \\
      &      &     &   1  &      &     &     &     &     \\
      &      &      &     &   1  &     &     &     &     \\
p & \times &  \times & \times &c_{40}&    &     &     &     \\
\times &\times &\times &c_{31}&      &   1 &    &     &     \\
\times &\times &c_{22}&      &      &     &   1 &    &     \\
\times &c_{13}&      &      &      &     &     &   1 &    \\
c_{04}&      &      &      &      &     &     &     &   1
\end{bmatrix}.
\]
It is easy to see that there exist permutation matrices $P_L$ and $P_R$ such that
\[
P_L(QMZ)P_R=
\begin{bmatrix}
1     &     &      &      &      &     &     &     &     \\
      &   1  &     &      &      &     &     &     &     \\
      &      &   1  &     &      &     &     &     &     \\
      &      &      &   1  &     &     &     &     &     \\
c_{40} & \times &  \times & \times &p&    &     &     &     \\
 &\times &\times &c_{31}&  \times    &   1 &    &     &     \\
 &\times &c_{22}&      & \times     &     &   1 &    &     \\
 &c_{13}&      &      & \times     &     &     &   1 &    \\
&      &      &      &  c_{04}    &     &     &     &   1
\end{bmatrix}
\]
is triangular and has the diagonal which satisfies
Theorem \ref{thm:matpollin}. Therefore, we can apply
\eqref{eq:repjan} for matrix polynomials by using block
matrices. This can be generalised to a uniform
representation of size $2d+1$ of the form
\eqref{eq:repjan}. In a similar way we can show that
this also holds for representations of the form
\eqref{eq:repjan2}  of size $2d-1$.
\hfill $\clubsuit$
}
\end{ex}

Unfortunately, not all uniform determinantal representations induce a determinantal
representation of a general matrix polynomial in this manner. As a counterexample,
let $M$ be such a uniform determinantal representation of the polynomial
$p_{n,d}$, $|\alpha|,|\beta| \le d$, and construct a representation of larger size
\[
M'=\begin{bmatrix}
M & 0\\
0 & N\\
\end{bmatrix},
\quad\quad \hbox{with}\quad N=
\left[
\begin{array}{rrrr}
0  & c_{\alpha}  & c_{\beta} & 1\\
-c_{\alpha} & 0  & 1 & 0\\
-c_{\beta} & -1 & 0 & 0\\
-1 & 0  & 0 & 1
\end{array}
\right].
\]
Then $\det(M')=\det(M)\det(N)=
p_{n,d}(1+c_{\alpha}c_{\beta}-c_{\beta}c_{\alpha})=p_{n,d}$, but
$\widetilde{M'}$ is not a representation for the matrix polynomial
$P_{n,d}$ as the coefficient matrices $C_{\alpha}$ and $C_{\beta}$
do not commute in general.
This motivates the following definition.

\begin{de}
A uniform determinantal representation $M$ is \emph{minimal} if there do not exist
constant matrices $P$ and $Z$ such that $\det(P)=\det(Z)=1$ and
\[
QMZ=\begin{bmatrix}
* & *\\
0 & M_2\\
\end{bmatrix},
\quad \quad \hbox{where $M_2$ is square with} \quad \det(M_2) =1.
\]
\end{de}
We speculate that each minimal uniform representation gives rise to a
representation for a matrix polynomial.

\section{Numerical experiments} \label{sec:Numerics}
Recently, a new numerical approach for computing roots of systems of bivariate
polynomials was proposed in \cite{BorMichiel}. The main idea is
to treat the system as a two-parameter
eigenvalue problem using determinantal representations.

Suppose that we are looking for roots of a system of bivariate polynomials
\begin{equation}\label{eq:syspol}
\begin{matrix}p=\sum_{i+j \le d_1} \alpha_{ij} x^i y^j=0,\\[0.2em]
q=\sum_{i+j \le d_2} \beta_{ij} x^i y^j=0,\end{matrix}
\end{equation} where
$p$ and $q$
are polynomials
of degree $d_1$ and $d_2$
over $\CC$. Let
$P=A_0 + x A_1 + y A_2$ and $Q=B_0 + x B_1 + y B_2$,
where $A_0,A_1,A_2 \in \CC^{N_1 \times N_1}$ and $B_0,B_1,B_2 \in \CC^{N_2 \times N_2}$, with $\det(P)=p$ and
$\det(Q)=q$, be determinantal representations of $p$ and $q$, respectively.
Then a root $(x,y)$ of \eqref{eq:syspol} is an
eigenvalue of the two-parameter eigenvalue problem
\begin{equation}\label{eq:twopar}
\begin{matrix}
(A_0+x A_1+y A_2) \, u=0,\\[0.2em]
(B_0+x B_1+y B_2) \, v=0,
\end{matrix}
\end{equation}
where $u\in\CC^{N_1}$ and $v\in\CC^{N_2}$ are nonzero vectors.
The standard way to solve \eqref{eq:twopar} is to
consider a joint pair
 of generalized eigenvalue problems  \cite{Atkinson}
\begin{equation}\label{eq:twopardelta}
\begin{matrix}
(\Delta_1-x \Delta_0) \, w=0,\\[0.2em]
(\Delta_2-y \Delta_0) \, w=0,
\end{matrix}
\end{equation}
where
\[
\Delta_0=A_1\otimes B_2-A_2\otimes B_1,\quad
\Delta_1=A_2\otimes B_0-A_0\otimes B_2,\quad
\Delta_2=A_0\otimes B_1-A_1\otimes B_0,
\]
and $w=u\otimes v$.

In this particular application we can expect that the
pencils in \eqref{eq:twopardelta} are singular, i.e., $\det(\Delta_1-x\Delta_0)\equiv 0$ and
$\det(\Delta_1-y\Delta_0)\equiv 0$. Namely, by
B\'ezout's theorem a generic system \eqref{eq:syspol}
has $d_1d_2$ solutions, while a generic problem \eqref{eq:twopar} has $N_1N_2$ eigenvalues.
Unless $(d_1,d_2)=(N_1,N_2)$, both pencils in \eqref{eq:twopardelta} are singular. In this
case we first have to apply the staircase algorithm from \cite{MuhicPlestenjakLAA}
to extract the finite regular eigenvalues. The method returns smaller matrices
$\widetilde \Delta_0$,
$\widetilde \Delta_1$, and $\widetilde \Delta_2$ (of size $d_1d_2\times d_1d_2$ for a generic \eqref{eq:syspol})
such that $\widetilde \Delta_0$ is nonsingular and ${\widetilde \Delta_0}^{-1}\widetilde \Delta_1$ and
${\widetilde \Delta_0}^{-1}\widetilde \Delta_2$ commute. From
\[
\begin{matrix}
(\widetilde\Delta_1-x \widetilde\Delta_0) \, \widetilde w=0,\\[0.2em]
(\widetilde\Delta_2-y \widetilde\Delta_0) \, \widetilde w=0,
\end{matrix}
\]
we compute the eigenvalues $(x,y)$ using a variant of the QZ algorithm \cite{HKP}
and thus obtain the roots of \eqref{eq:syspol}.

The above approach is implemented in the Matlab package BiRoots \cite{BorBR} together with
the two determinantal representations
from \cite{BorMichiel}. The first one, to which we refer as {\tt Lin1}, is a
uniform one from Example \ref{ex:lin1} of size $\frac{1}{4} \, d^2+ O(d)$ for a polynomial of degree $d$.
The second one, which we refer to as {\tt Lin2}, is not uniform and involves
some computation to obtain a smaller size $\frac{1}{6} \, d^2+ O(d)$.
Although the construction of {\tt Lin2} is more time consuming,
this pays off later, when the staircase algorithm is applied to \eqref{eq:twopardelta}.

Table~\ref{tbl:compare1} shows the sizes of determinantal representations for polynomials of
small degree.
As expected, the new uniform determinantal representation of size
$2d-1$, to which we refer as {\tt MinUnif}, returns smaller matrices, which reflects later
 in faster computational times. It is also
important that {\tt Lin1} and {\tt MinUnif} return real matrices for polynomials with real
coefficients, which is not true for {\tt Lin2}.

\begin{table}[!htbp]
\begin{footnotesize}
\begin{center}
\caption{Size of the matrices for {\tt Lin1} and {\tt Lin2}
for bivariate polynomials ($n=2$) and various degrees $d$.}\label{tbl:compare1}
\begin{tabular}{l|cccccccccc} \hline \rule{0pt}{2.3ex}%
Method & $d=3$ & $d=4$ & $d=5$ & $d=6$ & $d=7$ & $d=8$ & $d=9$ & $d=10$ & $d=11$ & $d=12$\\
\hline \rule{0pt}{2.3ex}%
{\tt Lin1} & 5 & 8 & 11 & 15 & 19 & 24 & 29  & 35 & 41 & 48  \\
{\tt Lin2} & 3 & 5 & \ph{1}8 & 10 & 13 & 17 & 20  & 24 & 29 & 34\\
{\tt MinUnif} & 5 & 7 & \ph{1}9 & 11 & 13 & 15 & 17  & 19 & 21 & 23\\
\hline
\end{tabular}
\end{center}
\end{footnotesize}
\end{table}

It was reported in
\cite{BorMichiel} that the determinantal representation approach for solving
systems of bivariate polynomials is competitive for polynomials of degree 9 or less.
As we show below, the new uniform representation {\tt MinUnif} extends
this to degree 15 and, in addition, performs
better than the existing representations for polynomials of degree 6 or more.

In \cite{BorMichiel} the approach was compared numerically to the following
state-of-the art numerical methods for polynomial systems:
{\tt NSolve} in Mathematica~9~\cite{Wolfram},
BertiniLab~1.4~\cite{Bertini} running Bertini~1.5~\cite{BertiniExe},
NAClab~3.0~\cite{NAClab},
and {\tt PHCLab}~1.04~\cite{PHClab} running PHCpack~2.3.84,
which turned out as the fastest of these methods.
To show the improved performance of the new determinantal representation,  we compare
{\tt MinUnif} to {\tt Lin1}, {\tt Lin2}, and {\tt PHCLab}
in Table~\ref{tbl:two}. For each $d$ we run the methods on the same set of
50 real and 50 complex random polynomial systems of degree $d$
and measure the average time.
For {\tt Lin1} and {\tt MinUnif}, where determinantal representations have real
matrices for real polynomials, we report separate results
for polynomials with real and complex coefficients. The timings
for {\tt Lin1} and {\tt Lin2} are given only for $n\le 10$ as
for larger $n$ these two linearisations are no longer competitive.

\begin{table}[!htbp]
\begin{footnotesize}
\begin{center}
\caption{Average computational times in milliseconds
for {\tt Lin1}, {\tt Lin2}, {\tt MinUnif}, and {\tt PHCLab} for
random full bivariate polynomial systems of degree $3$ to $15$. For {\tt Lin1} and {\tt MinUnif}
separate results are included for real $(\RR)$ and complex polynomials $(\CC)$.
}\label{tbl:two}
\begin{tabular}{r|rrrrrr} \hline \rule{0pt}{2.3ex}%
$d$ & {\tt Lin1} ($\RR$) & {\tt Lin1} ($\CC$) & {\tt Lin2} & {\tt PHCLab} & {\tt MinUnif} ($\RR$) & {\tt MinUnif} ($\CC$) \\
\hline \rule{0pt}{2.3ex}%
3 & \ph{111}6 & \ph{111}8 & \ph{11}4 & \ph{1}116 & \ph{11}6 & \ph{11}7 \\
4 & \ph{111}9 & \ph{11}11 & \ph{11}6 & \ph{1}130 & \ph{1}12 & \ph{11}13 \\
5 & \ph{11}20 & \ph{11}26 & \ph{1}13 & \ph{1}151 & \ph{1}18 & \ph{11}20 \\
6 & \ph{11}39 & \ph{11}71 & \ph{1}28 & \ph{1}174 & \ph{1}27 & \ph{11}27 \\
7 & \ph{11}96 & \ph{1}160 & \ph{1}51 & \ph{1}217 & \ph{1}36 & \ph{11}44 \\
8 & \ph{1}205 & \ph{1}395 & 118 & \ph{1}264 & \ph{1}59 & \ph{11}74   \\
9 & \ph{1}467 & 1124 & {279} & \ph{1}329 & \ph{1}{95} & \ph{1}125 \\
10 & 1424 & 3412 & 600 & \ph{1}414 & {147} & \ph{1}221  \\
11 &  &  &  & \ph{1}538 & {248} & \ph{1}354 \\
12 &  &  &  & \ph{1}650 & {361} & \ph{1}530 \\
13 &  &  &  & \ph{1}911 & {592} & \ph{1}740 \\
14 &  &  &  & 1142 & {842} & 1148  \\
15 &  &  &  & 1531 & {1237} & 1835  \\
\hline
\end{tabular}
\end{center}
\end{footnotesize}
\end{table}

Of course, the computational time is not the only important factor, we also have to
consider the accuracy and reliability. In each step of the staircase algorithm a
rank of a matrix has to be estimated numerically, which is a delicate task.
After several steps it may happen that the gap between the important singular values
and the meaningful ones that should be zero in exact computation, virtually disappears.
In such case the algorithm fails and does not return any roots.
As the number of steps in the staircase algorithm increases with degree of the polynomials,
such problems occur more often for polynomials of large degree.
 A heuristic that usually helps in such
cases is to apply the algorithm on a transformed system
\begin{align*}
\widetilde p &:= \ph{-}c p + s q = 0,\\
\widetilde q &:= -s p + c q = 0
\end{align*}
for random $c$ and $s$ such that $c^2+s^2=1$.
As this transformation does not change the conditioning of the roots, we
can conclude that the difficulties with the staircase algorithm are not directly related to
the conditioning. The trick does not work every time, and it seems that for
some systems the only way to make the determinantal representation approach to work is
to increase the machine precision.

We can apply the same approach to systems of polynomials in more than two variables.
However, since the size of the corresponding $\Delta$ matrices
is the product of sizes of all representations, this is competitive only
for $n=3$ and $d\le 3$. For a comparison, if we have a system of three polynomials
in three variables of degree $3$, then the size of the $\Delta$ matrices
is $343\times 343$. For degree $4$ the size increases to $1000\times 1000$
and PHCpack is faster. Finally, for $n=4$ and the smallest nontrivial $d=2$
we already get $\Delta$ matrices of size $625\times 625$ and the method
is not efficient.

\section{Outlook} \label{sec:Outlook}

We have introduced {\em uniform determinantal representations}, which
rather than representing a single polynomial as the determinant of a
matrix of affine-linear forms, represent all polynomials of degree at
most $d$ in $n$ variables as such a determinant.  We have seen that in
the bivariate case, these determinantal representations are useful for
numerically solving bivariate systems of equations; and in the general
multivariate case we have determined, up to constants, the asymptotic
behaviour of $N^*(n,d)$, the minimal size of such a representation,
for $n$ fixed and $d \to \infty$.

We now summarise several results that have been shown in the paper.
\begin{itemize}
\item For fixed $n$ and $d \to \infty$, $N^*(n,d) \sim d^{n/2}$, see
Theorem~\ref{thm:Main}.
This is a noticeable improvement on \cite{Quarez}, where an asymptotic rate
of $N^*(n,d) \sim d^n$ is shown, with the remark that the representation
in \cite{Quarez} are symmetric. However, symmetry currently cannot be exploited
by methods that compute roots of multivariate polynomial systems.
\item For fixed odd $d$ and $n \to \infty$, $N^*(n,d) \sim n^{(d-1)/2}$, which
is the same rate as Quarez \cite{Quarez}, who also manages to get symmetric representations.
For fixed even $d$ and $n \to \infty$, $N^*(n,d) \lsim n^{d/2}$, which
again is the same rate as in \cite{Quarez}. However, we have a slightly
smaller lower bound for the asymptotic rate of $n^{(d-1)/2}$.
\item Tables~\ref{tab:minnd} and \ref{tab:wv} give constructions
for the smallest representations that we have been able to find for some small values of $n$ and $d$.
\item $N^*(n,2) \le n+1$; cf.~Table~\ref{tab:minnd}.
\item $N^*(n,3) \le 2n+1$; cf.~Table~\ref{tab:minnd}.
\item $N^*(n,4) \le \frac{1}{6} n^2+O(n)$; see Example~\ref{ex:turan}.
\item $N^*(2,d) \le 2d-1$; cf.~Table~\ref{tab:minnd}.
Note that this result satisfies Dixon~\cite{Dixon}
up to an asymptotic factor 2 whereby no computations are necessary
for the determinantal representation.
In particular, it is a major improvement on the $\sim \frac{1}{4} d^2$
of \cite{Quarez, BorMichiel}.
\item Due to the smaller sizes of the representations, the numerical
approach for bivariate polynomials ($n=2$) is competitive to (say)
Mathematica for degree $d$ up to $d \approx 15$ (see Section~\ref{sec:Numerics});
this in contrast to $d \approx 9$ as obtained in \cite{BorMichiel}.
\item Under some conditions, the results carry over to the case
of matrix coefficients (see Section~\ref{sec:MatPol}).
\end{itemize}
There are still many interesting open questions, both of intrinsic mathematical
interest and of relevance to polynomial system solving.
First of all, in a situation where the degree $d$ is fixed and the number
$n$ of variables grows, what is the asymptotic behaviour of $N^*(n,d)$?
Although we expect that the above inequalities are equalities,
whether this holds is still open.
For general fixed $d$, the proof of Theorem~\ref{thm:Main} yields a lower
bound which is a constant (depending on $d$) times $n^{(d-1)/2}$. For odd
$d$ we obtain a matching upper bound (with a different constant) by using
Proposition~\ref{prop:Cons2} with $V=W=F_{(d-1)/2}$. However, for even $d$
we only know how to obtain $O(n^{d/2})$. We remark that this latter is
(up to a constant) the same bound as obtained in \cite[Thm.~4.4]{Quarez}
for {\em symmetric} uniform representations.

Second, in the case of fixed $n$ and varying $d$ studied in this paper,
what are the best constants in Theorem~\ref{thm:Main}? More specifically,
for fixed $n$, does $\lim_{d \to \infty} \frac{N^*(n,d)}{d^{n/2}}$ exist,
and if so, what is its value?

Third, how can our techniques for upper bounds and lower bounds be further
sharpened? Can singular matrix spaces {\em other} than compression spaces
be used to obtain tighter upper bounds (constructions) on $N^*(n,d)$?
Can the action of the affine group be used more systematically to find
lower bounds on $N^*(n,d)$?

Fourth, is it true that each minimal uniform representation gives
rise to a representation of the corresponding matrix polynomial
(cf.~Section~\ref{sec:MatPol})?

Finally, we have restricted our attention to matrices that, apart from
being affine-linear in $x_1,\dots,x_n$, are also affine-linear in the
coefficients $c_\alpha$. Our proofs give the same asymptotic behaviour
(with different constants) if we require, in addition, that no quadratic terms $c_\alpha
x_i$ may occur in $M$. If, instead, we relax the condition that
$M$ be affine-linear in the $c_\alpha$ to a {\em polynomial} dependence on
the $c_\alpha$, then the same bounds still apply; see
Remark~\ref{re:Poly}. But what if we relax this to
{\em rational} dependence of $M$ on the $c_\alpha$?  Given that $p_{n,d}$
is only linear in the $c_\alpha$ it seems unlikely that allowing $M$ to
be rational in the $c_\alpha$ we would gain anything, but we currently
do not know how to formalise this intuition. On the other hand, in
cases where a (non-uniform) determinantal representation of size $d$ is known to
exist for every (or sufficiently general) polynomials of degree $d$
in $d$ variables (e.g., in the case of plane curves), it follows that
this representation can be chosen to have entries {\em algebraic} in
the $c_\alpha$. This observation rules out approaches
aimed at proving lower bounds in a too general setting.


\vskip1in


\begin{thebibliography}{99}

\bibitem{Atkinson} {\sc F.~V.~Atkinson},
{\em Multiparameter Eigenvalue Problems}, Academic Press, New York (1972).

\bibitem{Beauville} {\sc A.~Beauville},
Determinantal hypersurfaces, Mich.~Math.~J.~48 (2000), 39--64.

\bibitem{BertiniExe} {\sc D.~J.~Bates, J.~H.~Hauenstein, A.~J.~Sommese, and C.~W.~Wampler},
{Bertini: Software for Numerical Algebraic Geometry}, available at \url{bertini.nd.edu}.

\bibitem{BoraleviFaenziMezzetti} {\sc A.~Boralevi, D.~Faenzi and E.~Mezzetti},
{Linear spaces of matrices of constant rank and instanton bundles},
Adv.~Math.~248 (2013), 895--920.

\bibitem{Bourbaki} {\sc N.~Bourbaki},
{\em Groupes et Alg\`{e}bres de {L}ie, Chapitres IV, V et VI}, \'El\'ements de Math\'ematique XXXIV,
Hermann, Paris (1968).

\bibitem{BurgisserIkenmeyerPanova}{\sc P.~B\"urgisser, C.~Ikenmeyer, and G.~Panova},
No occurrence obstructions in geometric complexity theory,
\verb+arXiv:1604.06431+ (2016).

\bibitem{Branden} {\sc P.~Br\"and\'en},
Obstructions to determinantal representability,
Adv.~Math.~226(2), (2011), 1202--1212.

\bibitem{Dickson} {\sc L.~E.~Dickson},
{Determination of all general homogeneous polynomials expressible as determinants with linear elements},
Trans.~Amer.~Math.~Soc.~22 (1921), 167--179.

\bibitem{Dieudonne} {\sc J.~Dieudonn\'e},
{Sur une g\'en\'eralisation du groupe orthogonal \`a quatre variables},
Arch.~Math., Oberwolfach (1949), 1:282–287.

\bibitem{Dixon} {\sc A.~Dixon},
{Note on the reduction of a ternary quartic to a symmetrical determinant},
Proc.~Camb.~Phil.~Soc.~11 (1902), 350--351.

\bibitem{DraismaBLMS} {\sc J.~Draisma},
{Small maximal spaces of non-invertible matrices},
Bull.~Lond.~Math.~Soc.~38(5) (2006), 764--776.

\bibitem{EisenbudHarris} {\sc D.~Eisenbud and J.~Harris},
{Vector spaces of matrices of low rank},
Adv.~Math.~70(2) (1988), 135--155.

\bibitem{FillmoreLaurieRadjavi} {\sc P.~Fillmore, C.~Laurie, and H.~Radjavi},
{On matrix spaces with zero determinant},
Lin.~Multilin.~Algebra, 18 (1985), 255--266.

\bibitem{Flanders} {\sc H.~Flanders},
{On spaces of linear transformations with bounded rank},
J.~Lond.~Math.~Soc.~37 (1962), 10--16.

\bibitem{PHClab} {\sc Y.~Guan and J.~Verschelde},
PHClab: A MATLAB/Octave interface to PHCpack.
In: M.~Stillman, J.~Verschelde, and N.~Takayama (eds), {\sl Software for Algebraic Geometry},
volume 148 of The IMA Volumes in Mathematics and its Applications, Springer,
New York (2008), 15--32.

\bibitem{HeltonVinnikov} {\sc J.~W.~Helton and V.~Vinnikov},
Linear matrix inequality representation of sets,
Commun.~Pure~Appl.~Math.~60(5) (2007), 654--674.

\bibitem{HKP} {\sc M.~E.~Hochstenbach, T.~Ko\v{s}ir, and B.~Plestenjak},
{A Jacobi--Davidson type method for the nonsingular two-parameter eigenvalue problem},
SIAM J.~Matrix Anal.~Appl.~26 (2005), 477--497.

\bibitem{HuttenhainLairez} {\sc J.~H\"uttenhain and P.~Lairez},
The boundary of the orbit of the 3 by 3 determinant polynomial,
\verb+arXiv:1512.02437+ (2015).

\bibitem{IlicLandsberg} {\sc B.~Ilic and J.~M.~Landsberg},
{On symmetric degeneracy loci, spaces of symmetric matrices of constant rank and dual varieties},
Math.~Ann.~314(1) (1999), 159--174.

\bibitem{IshitshukaIto} {\sc Y.~Ishitsuka and T.~Ito},
On the symmetric determinantal representations of the Fermat curves of prime degree,
Int.~J.~Number Theory 12(4) (2016), 955--967.

\bibitem{LLMRT} {\sc
J-B.~Lasserre, M.~Laurent, B.~Mourrain, Ph.~Rostalski, and Ph.~Tr\'ebuchet},
Moment matrices, border bases and real radical computation,
J.~Symb.~Comput.~51 (2013), 63--85.

\bibitem{LewisParriloRamana} {\sc A.~S.~Lewis, P.~A.~Parrilo, and M.~V.~Ramana},
The Lax conjecture is true,
Proc.~Am.~Math.~Soc.~133(9), (2005), 2495--2499.

\bibitem{ManivelMezzetti} {\sc L.~Manivel and E.~Mezzetti},
{On linear spaces of skew-symmetric matrices of constant rank},
Manuscripta Math.~117(3) (2005), 319--331.

\bibitem{Matlab} {\sc The MathWorks, Inc.}, Matlab, Natick, Massachusetts, United States.

\bibitem{MuhicPlestenjak09} {\sc A.~Muhi\v{c} and B.~Plestenjak},
{On the singular two-parameter eigenvalue problem},
Electron.~J.~Linear Algebra 18 (2009), 420--437.

\bibitem{MuhicPlestenjakLAA} {\sc A.~Muhi\v{c} and B.~Plestenjak},
{On the quadratic two-parameter eigenvalue problem and its linearization},
Linear Algebra Appl.~432 (2010), 2529--2542.

\bibitem{MulmuleySohoniI}{\sc K.~D.~Mulmuley and M.~Sohoni},
Geometric complexity theory.~I: An approach to the P vs.~NP and related problems,
SIAM J.~Comput.~31(2) (2001), 496--526.

\bibitem{MulmuleySohoniII}{\sc K.~D.~Mulmuley and M.~Sohoni},
Geometric complexity theory II: Towards explicit
obstructions for embeddings among class varieties,
SIAM J.~Comput.~38(3) (2008), 1175--1206.

\bibitem{Bertini} {\sc A.~Newell},
BertiniLab: toolbox for solving polynomial systems, MATLAB Central File Exchange,
\url{www.mathworks.com/matlabcentral/fileexchange/48536-bertinilab}.

\bibitem{BorBR} {\sc B.~Plestenjak}, BiRoots, MATLAB Central File Exchange,
\url{www.mathworks.com/matlabcentral/fileexchange/54159-biroots} (2015).

\bibitem{BorMichiel} {\sc B.~Plestenjak and M.~E.~Hochstenbach},
Roots of bivariate polynomial systems via determinantal representations,
SIAM J.~Sci.~Comput.~38 (2016), A765--A788.

\bibitem{Robol} {\sc L.~Robol, R.~Vandebril, and P.~van Dooren},
{A framework for structured linearizations of matrix polynomials in various bases},
\verb+arXiv:1603.05773+ (2016).

\bibitem{Quarez} {\sc R.~Quarez},
Symmetric determinantal representation of polynomials,
Linear Algebra Appl.~436(9) (2012), 3642--3660.

\bibitem{Sidorenko} {\sc A.~Sidorenko}, {What we know and what we
do not know about Tur\'{a}n numbers}, Graphs Combin.~11 (1995), 179--199.


\bibitem{deSeguinsPazzis} {\sc C.~de~Seguins~Pazzis},
Large affine spaces of matrices with rank bounded below,
Linear Algebra Appl.~437(2) (2012), 499--518.

\bibitem{Sylvester}	{\sc J.~Sylvester},
{On the dimension of spaces of linear transformations satisfying rank conditions},
Linear Algebra Appl.~78 (1986), 1--10.

\bibitem{Valiant} {\sc L.~G.~Valiant}, {The complexity of computing the permanent},
Theoret.~Comput.~Sci.~8(2) (1979), 189--201.

\bibitem{phc} {\sc J.~Verschelde},
{Algorithm 795: PHCpack: a general-purpose solver for polynomial systems by homotopy continuation},
ACM Trans.~Math.~Softw., 25 (1999), 251--276.

\bibitem{Wagner} {\sc D.~G.~Wagner},
Multivariate stable polynomials: theory and applications,
Bull.~Am.~Math.~Soc., New Ser.~48(1), (2011), 53--84.

\bibitem{Westwick} {\sc R.~Westwick}, {Spaces of matrices of fixed rank},
Lin.~Multilin.~Algebra 20 (1987), 171--174.

\bibitem{Wolfram} {\sc Wolfram Research, Inc.}, Mathematica, Version 9.0,
Champaign, Illinois (2012).

\bibitem{NAClab} {\sc Z.~Zeng and T.-Y.~Li},
NAClab: a Matlab toolbox for numerical algebraic computation,
ACM Commun.~Comput.~Algebra~47 (2013), 170--173.

\end{thebibliography}
\end{document}